\documentclass[a4paper,twoside]{article}
\usepackage[utf8x]{inputenc}		

\usepackage{amsthm,amscd}
\usepackage{amsmath,amssymb,amsfonts}
\usepackage{version}

\usepackage[colorlinks=true,citecolor=blue,linkcolor=blue]{hyperref}
\usepackage{stmaryrd}				
\usepackage{bbm}					
\usepackage{version}
\usepackage{geometry}
\newtheorem{theorem}{Theorem}[section]
\newtheorem{corollary}[theorem]{Corollary}
\newtheorem{lemma}[theorem]{Lemma}
\newtheorem{proposition}[theorem]{Proposition}

\theoremstyle{definition}

\newtheorem{remark}[theorem]{Remark}

\newcommand{\ensnombre}[1]{\mathbb{#1}}
\newcommand{\N}{\ensnombre{N}}
\newcommand{\Z}{\ensnombre{Z}}

\newcommand{\R}{\ensnombre{R}}

\newcommand{\defeq}{\mathrel{\mathop:}=}
\newcommand{\eqdef}{\mathrel{=}:}
\newcommand{\indic}{\mathbbm{1}}
\newcommand{\Prob}{\mathbb{P}}
\newcommand{\Espe}{\mathbb{E}}
\DeclareMathOperator{\Var}{Var}
\DeclareMathOperator{\Cov}{Cov}
\newcommand{\intervalle}[4]{\left #1 #2, #3 \right #4}
\newcommand{\intervalleff}[2]{\intervalle{[}{#1}{#2}{]}}
\newcommand{\intervalleof}[2]{\intervalle{]}{#1}{#2}{]}}
\newcommand{\intervallefo}[2]{\intervalle{[}{#1}{#2}{[}}
\newcommand{\intervalleoo}[2]{\intervalle{]}{#1}{#2}{[}}
\newcommand{\intervallentff}[2]{\intervalle{\llbracket}{#1}{#2}{\rrbracket}}
\newcommand{\abs}[1]{\left\lvert #1 \right\rvert}
\newcommand{\norme}[1]{\left\lVert #1 \right\rVert}
\newcommand{\leqp}{\leqslant}
\newcommand{\geqp}{\geqslant}

\newcommand{\chap}{^{(n)}}
\newcommand{\ET}[1]{\sigma_{#1}}
\newcommand{\RHO}[1]{\rho_{#1}}
\newcommand{\bnxm}{\sigma_{X\chap}N_n^{1/2}}
\newcommand{\bnym}{\sigma_{Y\chap}N_n^{1/2}}
\newcommand{\bnx}{\sigma_{X\chap}^{-1}N_n^{-1/2}}
\newcommand{\bny}{\sigma_{Y\chap}^{-1}N_n^{-1/2}}

\title{A conditional Berry-Esseen bound and a conditional\\ large deviation result without Laplace transform.\\ Application to hashing with linear probing.}
\author{T. Klein, A. Lagnoux, P. Petit\\
Institut de Math\'ematiques, University of Toulouse}
\begin{document}

\maketitle


\begin{abstract}
\noindent We study the asymptotic behavior of a sum of independent and identically distributed random variables conditioned by a sum of independent and identically distributed integer-valued random variables.
We prove a Berry-Esseen bound in a general setting and a large deviation result when the Laplace transform of the  underlying distribution is not defined in a neighborhood of zero. Then we present several combinatorial applications. In particular, we prove a large deviation result for the model of hashing with linear probing.


\end{abstract}

\section{Introduction}

As pointed out by Svante Janson in his seminal work \cite{Janson01}, in many random combinatorial problems, the interesting statistic is the sum of independent and identically distributed (i.i.d.)\ random variables conditioned by some exogenous integer random variable. In general, this exogenous random variable is itself a sum of integer-valued random variables. A general framework for this kind of problem may be formalized as follows. In the whole paper, $\N^*$ will denote the set $\{1, 2, \ldots \}$ of positive integers, $\N = \N^* \cup \{ 0 \}$, and $\Z$ will be the set of all integers. Let $(k_n)_{n\in\N^{*}}$ be a sequence of integers  and $(N_n)_{n\in\N^{*}}$ be a sequence of  positive integers. Further, let $(X_j^{(n)},Y_j^{(n)})_{n\in\N^*, j=1,\ldots,N_n}$ be a triangular array of pairs of random variables such that each line contains i.i.d.\ copies of a pair $(X\chap , Y\chap)$ of random variables. Moreover, it is assumed that the elements of the array $(X_j^{(n)})_{n\in\N^*, j=1,\ldots,N_n}$ are integers. We are interested in the law of $(N_n)^{-1}T_{n} \defeq (N_n)^{-1} \sum_{j=1}^{N_n}Y_j^{(n)}$ conditioned on a specific value of $S_n \defeq \sum_{j=1}^{N_n}X_j^{(n)}$; that is to say in the conditional distribution
\[
\mathcal{L}_n \defeq \mathcal{L}((N_n)^{-1}T_n|S_n=k_n).
\]
The motivation for considering distributions of $(X^{(n)},Y^{(n)})$ that depend on $n$ comes from the discrete nature of the problem
that can lead to a degenerated conditional law as soon as $\Prob(S_n=k_{n})=0$. Nevertheless in many applications (e.g., occupancy problem or hashing ; see \cite{Janson01}), the distribution of the conditioning random variable $X$ depends on a parameter $\lambda$ that can be freely chosen: for example, $\lambda\in\R$  is the parameter of a Poisson distribution in the occupancy problem and $\lambda\in]0,e^{-1}]$ is the parameter of the Borel distribution for hashing. One can take advantage of this fact to overcome contexts in which $\Prob(S_n=k_{n})=0$ proceeding as follows. Consider a triangular array  $(X_j^{(n)},Y_j^{(n)})_{n\in \N^*, j=1\ldots N_n}$ such that  $(X^{(n)}, Y^{(n)})$ converges weakly to $(X, Y)$. Then choose a sequence of parameters $\lambda_{n} \to \lambda$ such that, for any $n$,  $\Prob(\sum_{j=1}^{N_n} X_j^{(n)} = k_{n})>0$.

In his work, Janson proves a general central limit theorem (with convergence of all moments) for this kind of conditional distribution under some reasonable assumptions and gives several applications in classical combinatorial problems: occupancy in urns, hashing with linear probing, random forests, branching processes, etc. Following this work, at least two  natural questions arise:
\begin{enumerate}
\item is it possible to obtain a general Berry-Esseen bound for these models?
\item is it possible to obtain a general large deviation result for these models?
\end{enumerate}
A Berry-Esseen theorem is given by Quine and Robinson \cite{QR82}. In their work, the authors study the particular case of the occupancy problem where the random variables $X^{(n)}$ are Poisson distributed and $Y^{(n)} = \indic_{\{ X^{(n)} = 0 \} }$. Up to our knowledge, it is the only result in that direction for this kind of conditional distribution. In our work, we prove a general Berry-Esseen bound (Theorem \ref{th:BE_cond_strong}) that covers all the examples presented by Janson  \cite{Janson01}.

When the distribution of $(X^{(n)}, Y^{(n)})$ does not depend on $n$, the Gibbs conditioning principle (\cite{VanC,csiszar1984,DZ98}) states that ${\mathcal{L}}_n$ converges weakly to the degenerated distribution concentrated on a point $\chi$ depending on the conditioning value (see \cite[Corollary 2.2]{TFC12}). Around the Gibbs conditioning principle,     
general limit theorems yielding the asymptotic behavior of the conditioned sum are given in \cite{Steck57,Holst79,Kud84} and asymptotic expansions are proved in \cite{Hipp84,Rob90}. In this paper our aim is to prove a large deviation result for $\mathcal{L}_n$, when the joint Laplace transform of $(X_j^{(n)}, Y_j^{(n)})$ is not defined everywhere: we give an exponential equivalent for this conditional distribution.

The case when the Laplace transform is defined has been treated by Gamboa, Klein and Prieur  \cite{TFC12}. They  prove a large (and a moderate) deviation principle under some strong assumptions. The most restricting assumption states that the joint Laplace transform of $(X^{(n)}, Y^{(n)})$ is finite at least in a neighborhood of $(0,0)$. Unfortunately, this assumption fails to be satisfied for the most interesting  example presented in \cite{Janson01}: hashing with linear probing. In this case, the joint Laplace transform is only defined in $]-\infty,a]\times ]-\infty,0]$ for some positive $a$. It is then natural to extend the work of \cite{TFC12} for such distributions. In  \cite{Nagaev69-1,Nagaev69-2}, Nagaev establishes large deviation results for sums of random variables which are absolutely continuous with respect to the Lebesgue measure and the Laplace transform of which is not defined in a neighborhood of $0$. Following this work, we prove a large deviation result (Theorem \ref{th:nagaev_weak_array_cond_mob_cor}).

Let us point out the main differences between Theorem \ref{th:nagaev_weak_array_cond_mob_cor} of the present work and Theorem 2.1 of \cite{TFC12}. First, the proof in \cite{TFC12} is based on a sharp control of a Fourier-Laplace transform $\Phi_{X^{(n)},Y^{(n)}}(t,u) \defeq \Espe\left(\exp[itX^{(n)}+u Y^{(n)}]\right)$ of $\left(X^{(n)}, Y^{(n)}\right)$. The Fourier part allows to treat the conditioning  whereas the Laplace one allows to apply Gärtner-Ellis theorem. In the present paper, the proof follows ideas borrowed from \cite{Nagaev69-1,Nagaev69-2}. More precisely, contrary to the case when the Laplace transform is defined, the large deviations of the sum of the random variables with heavy-tailed distributions is due to exceptional values taken by few random variables. Second, unlike the classical speeds in $N_{n}$ obtained either in Cramér's theorem or in Theorem 2.1 of \cite{TFC12}, the speed in this paper is $\sqrt{N_n}$. Third, one originality of our work is that the lower and upper bounds may differ (see equations \eqref{hyp:Yqueue1_weak_array_cond_mob_cor} and \eqref{hyp:Yqueue2_weak_array_cond_mob_cor}). When the Laplace transform is defined, the tails are controlled (see Cramér's theorem or Gärtner-Ellis theorem in \cite{DZ98}) and the sum satisfies a large deviation principle with the same lower and upper bounds. Here, as opposed to previous classical theorems, one may allow oscillations of the tails (in a controlled range) that lead to a large deviation result with two different bounds. Last but not least, the rate function obtained is not affected by the conditioning variable: the rate functions are the same in the conditional case and in the unconditional one (see Theorems \ref{th:nagaev_weak_array_cond_mob_cor} and \ref{th:nagaev_weak_array_mob_v2}). On the contrary, when the Laplace transform is defined in a neighborhood of the origin, the rate function strongly depends on the dependence between $X\chap$ and $Y\chap$. It is $y\mapsto\psi^{*}_{X\chap,Y\chap}(\lambda, y)-\psi^{*}_{X\chap}(\lambda)$ (where $\lambda$ is the limit of the ratio $k_n/N_n$), the difference between the joint Fenchel-Legendre transform and the Fenchel-Legendre transform of the conditioning random variable $X\chap$. This rate function is $y\mapsto\psi^{*}_{Y\chap}(y)$ when the conditioning term is ineffective, that is to say when the random variables $X\chap$ and $Y\chap$ are independent.

As pointed out by Janson in \cite{Janson01}, hashing with linear probing was the motivating example for his work (see section \ref{sec:hash} for a complete description of the model). This model comes from theoretical computer science, where it modelizes the time cost to store data in the memory. Then, it was introduced in a mathematical framework by Knuth \cite{Knuth98}. Due to its strong connection with parking functions, the Airy distributions (i.e., the area under the brownian excursion), this model was studied by many authors (see, e.g., Flajolet, Poblete and Viola \cite{FPV98}, Janson \cite{Janson01a,Janson05,Janson08}, Chassaing, Janson, Louchard and Marckert \cite{Chassaing01,Chassaing02,Marckert01-2}, and Marckert  \cite{Marckert01-1}). Theorem \ref{th:nagaev_weak_array_cond_mob_cor} allows to treat the interesting example of hashing with linear probing:
Proposition \ref{prop:hash_cond} is the formulation of Theorem \ref{th:nagaev_weak_array_cond_mob_cor} in this particular framework.

The paper is organized as follows. In section \ref{sec:main}, we present the general model and give our two main theorems. First we prove a Berry-Esseen bound (Theorem \ref{th:BE_cond_strong}) and show how it straightforwardly applies to the examples presented by Janson \cite{Janson01}. Second we establish a large deviation result (Theorem \ref{th:nagaev_weak_array_cond_mob_cor}). Section \ref{sec:hash} is devoted to the study of hashing with linear probing. 
Finally, we prove our main results in the last section.

\section{Main results}\label{sec:main}

\subsection{Framework and notation}\label{subsec:notation}

For all $n \geqslant 1$, we consider a pair of random variables $\left(X^{(n)},Y^{(n)}\right)$ such that $X^{(n)}$ is integer-valued and $Y^{(n)}$ real-valued.
Let $N_n$ be a natural number such that $N_n \to +\infty$ as $n$ goes to infinity.
Let $\left(X_i^{(n)},Y_i^{(n)}\right)$ ($i=1, 2, \ldots, N_n$) be an i.i.d.\ sample distributed as $\left(X^{(n)},Y^{(n)}\right)$ and define 
\[
S_n\defeq\sum_{i=1}^{N_n} X_i^{(n)} \quad \text{and} \quad  T_n\defeq\sum_{i=1}^{N_n} Y_i^{(n)}.
\]
Let $k_n \in \Z$ be such that $\Prob(S_n = k_n) > 0$ and let $U_n$ be a random variable distributed as $T_n$ conditioned on $S_n = k_n$. We establish a Berry-Esseen bound and a large deviation result for $(U_n)_{n \geqslant 1}$.

\subsection{Conditional Berry-Esseen bound}

\begin{theorem}\label{th:BE_cond_strong}
Suppose that there exist positive constants $\tilde{c}_1$, $c_1$, $c_2$, $\tilde{c}_3$, $c_3$, $c_4$, $c_5$, and $c_6$ such that:
\renewcommand{\theenumi}{\upshape\sffamily(H\thetheorem{}.\arabic{enumi})}
\renewcommand{\labelenumi}{\theenumi}
\begin{enumerate}
\addtolength{\itemindent}{2em}
\item \label{ass:var_X} $\tilde{c}_1 \leqslant \ET{X\chap} \defeq \Var\left(X\chap\right)^{1/2} \leqslant c_1$;
\item \label{ass:rho_X} $\RHO{X\chap} \defeq \Espe\left[\left|X\chap-\Espe\left[X\chap\right]\right|^3\right] \leqslant c_2^3 \ET{X\chap}^3$;
\item \label{ass:fc_X} define $Y^{'(n)} \defeq Y\chap - X\chap \Cov(X\chap, Y\chap)/\ET{X\chap}^2$, there exists $\eta_0>0$ such that, for all $s \in \intervalleff{-\pi}{\pi}$ and $t \in \intervalleff{0}{\eta_0}$,
\[
\abs{\Espe\left[ e^{i(sX\chap + tY^{'(n)})} \right]} \leqslant 1 - c_5 \big( \ET{X\chap}^2 s^2 + \ET{Y^{'(n)}}^2 t^2 \big) ;
\]
\item \label{ass:exp_X} $k_n = N_n\Espe\left[X\chap\right]+O(\ET{X\chap} N_n^{1/2})$ (remind that $k_n\in \Z$ and $\Prob(S_n=k_n)>0$);
\item \label{ass:var_Y} $\tilde{c}_3 \leqslant \ET{Y\chap} \defeq \Var\left(Y\chap\right)^{1/2} \leqslant c_3$;
\item \label{ass:rho_Y} $\RHO{Y\chap} \defeq \Espe\left[\left|Y\chap-\Espe\left[Y\chap\right]\right|^3\right] \leqslant c_4^3 \ET{Y\chap}^3$;
\item \label{ass:corr} the correlation $r_n \defeq \Cov\left(X^{(n)},Y^{(n)}\right) \ET{X\chap}^{-1} \ET{Y\chap}^{-1}$ satisfies $|r_n| \leqslant c_6 < 1$, so that
\[
\tau_n^2 \defeq \ET{Y\chap}^2 (1-r_n^2) \geqslant \tilde{c}_2^2 (1 - c_6^2) > 0.
\]
\end{enumerate}
\renewcommand{\theenumi}{\thetheorem{}.\alph{enumi}}
\renewcommand{\labelenumi}{\theenumi.}
Then the following conclusions hold.
\begin{enumerate}
\addtolength{\itemindent}{2em}
\item \label{th:BEa} There exists $\tilde{c}_5 > 0$ such that
\[
\Prob(S_n = k_n) \geqslant \frac{\tilde{c}_5}{2\pi \ET{X\chap} N_n^{1/2}}.
\]
\item \label{th:BEb} For $N_n \geqslant N_0 \defeq \max(3, c_2^6, c_4^6)$, the conditional distribution of 
\[
N_n^{-1/2}\tau_n^{-1} (T_n-N_n \Espe[Y^{(n)}]-r_n\frac{\ET{Y\chap}} {\ET{X\chap}}(k_n - N_n\Espe[X\chap]))
\]
given $S_n = k_n$ satisfies the Berry-Esseen inequality
\begin{equation}\label{eq:be}
\sup_x \left|\Prob\left(\frac{U_n-N_n \Espe\left[Y^{(n)}\right]-r_n \ET{Y\chap}\ET{X\chap}^{-1}(k_n - N_n\Espe\left[X\chap\right])}{N_n^{1/2}\tau_n}\leqp x \right)-\Phi(x)\right|\leqp \frac{C}{N_n^{1/2}},
\end{equation}
where $\Phi$ denotes the standard normal probability distribution, and $C$ is a positive constant that only depends on $\tilde{c}_1$, $c_1$, $c_2$, $\tilde{c}_3$, $c_3$, $c_4$, $c_5$, $\tilde{c}_5$, and $c_6$.
\item \label{th:BEc} Moreover, there exist two positive constants $c_7$ and $c_8$ only depending on $\tilde{c}_1$, $c_1$, $c_2$, $\tilde{c}_3$, $c_3$, $c_4$, $c_5$, $\tilde{c}_5$, and $c_6$ such that
\begin{equation}\label{eq:moment}
\abs{\Espe\left[U_n\right] - N_n \Espe[Y^{(n)}]-r_n\frac{\ET{Y\chap}} {\ET{X\chap}}(k_n - N_n\Espe[X\chap])} \leqslant c_7
\end{equation}
and
\begin{equation}\label{eq:moment2}
\abs{\Var\left(U_n\right) - N_n \tau_n^2} \leqslant c_8 N_n^{1/2}
\end{equation}
If $N_n \geqslant \tilde{N}_0 \defeq \max(N_0, 4 c_8^2/\tilde{c}_3^2)$, we also have
\begin{equation}\label{eq:be_2}
\sup_x \left|\Prob\left(\frac{U_n- \Espe\left[U_n\right]}{\Var\left(U_n\right)^{1/2}}\leqp x \right)-\Phi(x)\right|\leqp \frac{\tilde{C}}{N_n^{1/2}},
\end{equation}
where $\tilde{C}$ is a constant that only depends on $\tilde{c}_1$, $c_1$, $c_2$, $\tilde{c}_3$, $c_3$, $c_4$, $c_5$, $\tilde{c}_5$, and $c_6$. This result means that $U_n$ is asymptotically normal.
\end{enumerate}
\renewcommand{\theenumi}{\arabic{enumi}}
\renewcommand{\labelenumi}{\theenumi.}
\end{theorem}

\begin{remark}\label{rem:ass}\leavevmode
\begin{enumerate}
\item The fact that $N_n \to +\infty$ is only required for the existence of the constant $\tilde{c}_5$ which relies on Lebesgue dominated convergence theorem.
\item The set of hypotheses of Theorem \ref{th:BE_cond_strong} implies the one of the central limit theorem stated in \cite[Theorem 2.1]{Janson01} which is clearly not surprising.
Notice that by assumption \ref{ass:exp_X}, the conditioning is approximately equal to the mean as in the central limit theorem given in \cite[Theorem 2.3]{Janson01}.
\item As a consequence of Proposition \ref{sigmaXN} below, $\tilde{c}_1$ can be chosen as $c_2^{-3}/4$.
\item Assumption \ref{ass:corr} is not very restricting as we will see later in the examples.
\item One should note that 2.1.a is the analogue of Equation (7) of Lemma 3.2 in \cite{TFC12}.
\item In the proof, we will replace $Y^{(n)}$ by the projection $Y^{'(n)}$ in order to work with a centered variable which is also uncorrelated with $X^{(n)}$.
We introduce $Y^{'(n)}$ for that purpose.
\item If $(X, Y')$ is a pair of random variables such as the correlation $r$ satisfies $\abs{r} < 1$, then
\begin{align*}
\abs{\Espe\Big[ e^{i(sX + tY')} \Big]} & = 1 - \frac{1}{2} \big( \ET{X}^2 s^2 + 2 \ET{X} \ET{Y'} r s t + \ET{Y'}^2 t^2 \big) + o(s^2 + t^2) \\
 & \leqslant 1 - \frac{1-\abs{r}}{2} \big( \ET{X}^2 s^2 + \ET{Y'}^2 t^2 \big) + o(s^2 + t^2),
\end{align*}
so hypothesis \ref{ass:fc_X} is reasonable for i.i.d.\ sequences.
\end{enumerate}
\end{remark}

As mentioned in \cite{Janson01}, the result simplifies considerably in the special case when the pair $(X\chap, Y\chap)$ does not depend on $n$, that is to say when we consider a single sequence instead of a triangular array. This is a consequence of the following more general corollary.
\begin{corollary}\label{cor:va_iid}
Assume that $\left( X\chap, Y\chap \right) \overset{(d)}{\to} (X, Y)$ as $n \to \infty$ and that, for every fixed $r > 0$,
\[
\limsup_{n \to +\infty} \Espe\left[|X\chap|^r\right] < \infty \quad \text{and} \quad \limsup_{n \to +\infty} \Espe\left[|Y\chap|^r\right ]< \infty.
\]
Suppose further that the distribution of $X$ has span 1 and that $Y$ is not a.s. equal to an affine function $c+dX$ of $X$, that $k_n$ and $N_n$ are integers such that $\Espe\left[X\chap\right] = k_n/N_n$ and $N_n\to+\infty$. Then, all hypotheses of Theorem \ref{th:BE_cond_strong} are satisfied and Theorem \ref{th:BE_cond_strong} holds.
\end{corollary}

\subsection{Applications}

In this section we give several examples borrowed from \cite{Janson01} and \cite{Holst79}. A direct application of Corollary \ref{cor:va_iid} leads to Berry-Esseen bounds in each of them.

\subsubsection{Occupancy problem}\label{exocc}

In the classical occupancy problem (see \cite{Janson01} and the references therein for more details), $m$ balls are distributed at random into $N$ urns. The resulting numbers of balls $(Z_1, \ldots, Z_N)$ have a multinomial distribution which equals that of $\left( X_1,\cdots,X_N\right)$ conditioned on $\sum_{i=1}^N X_i=m$, where $X_1$, ..., $X_N$ are i.i.d.\ with $X_i \sim \mathcal{P}(\lambda),$ for any arbitrary $\lambda>0$. The classical occupancy problem studies the number $W$ of empty urns that is the distribution of $\sum_{i=1}^N \indic_{\{X_i=0\}}$ conditioned on $\sum_{i=1}^N X_i=m$.

Let us follow the work of Janson \cite{Janson01} and suppose that $m=k_n\to\infty$ and $N=N_n\to\infty$ with $k_n/N_n\to\lambda$. Then $W$ can be taken as $U_n$ in Theorem \ref{th:BE_cond_strong} with $X\chap \sim \mathcal{P}(\lambda_n)$ and $Y\chap = \indic_{\{X\chap=0\}}$ for any $\lambda_n$; we choose $\lambda_n=k_n/N_n$ so that assumption \ref{ass:exp_X} holds.

\begin{itemize}
\item If $k_n$, $N_n\to \infty$ such that $k_n/N_n\to \lambda\in (0,\infty)$, then Corollary \ref{cor:va_iid} immediately yields that the conclusions of Theorem \ref{th:BE_cond_strong} hold.
\item In the case $k_n/N_n\to \infty$, assumption \ref{ass:var_X} is clearly violated and Theorem \ref{th:BE_cond_strong} does not apply.
\item In the case $k_n/N_n\to 0$, Theorem \ref{th:BE_cond_strong} can not be applied as stated since $Y\chap = \indic_{\{X\chap=0\}}$ implies that assumption \ref{ass:corr} does not hold $(r_n\to -1)$. As explained in \cite{Janson01}, one can choose instead $Y\chap \defeq \indic_{\{X\chap=0\}}+X\chap-1=(X\chap-1)_+$ and it is clearly verified that Theorem \ref{th:BE_cond_strong} 
applies without any extra assumption.
\end{itemize}

\subsubsection{Branching processes}

Consider a Galton-Watson process, beginning with one individual, where the
number of children of an individual is given by a random variable $X$ having
finite moments. Assume further that $ \Espe(X)=1$. We number the individuals as
they appear. Let $X_i$ be the number of children of the $i^{\textrm{th}}$ individual. It is well known (see \cite[Example 3.4]{Janson01} and the references therein) that the total progeny is $n\geq 1$ if and only if
\begin{equation}
S_k\defeq\sum_{i=1}^kX_i\geq k\mbox{ for } 0\leq k<n \mbox{ but }S_n=n-1 \, .
\label{GW1}
\end{equation}
This type of conditioning is different from the one studied in the present
paper, but
Janson proves \cite[Example 3.4]{Janson01} that if we ignore the order of $X_1,
\ldots, X_n$, they have the same distribution conditioned on (\ref{GW1}) as conditioned
on $S_n=n-1$.
Hence our results apply to variables of the kind $Y_i=f(X_i)$. For example if $Y_i = \indic_{\{X_i=3\}}$,
the $\sum_{i=1}^n Y_i$ is the number of families with three children. 

\subsubsection{Random forests}

Consider a uniformly distributed random labeled rooted forest with $m$ vertices and $N < m$ roots. Without loss of generality, we may assume that the vertices are $1, \ldots, m$ and, by symmetry, that the roots are the first $N$ vertices. Following \cite{Janson01}, this model can be realized as follows: the sizes of the $N$ trees in the forest are distributed as $X_1, \ldots , X_N$ conditioned on $\sum_{i=1}^N X_i =m$, where $X_i$ are i.i.d.\ with the Borel distribution for some arbitrary parameter $\lambda \in \, ]0, 1/e]$ (see section \ref{subsec:hash_proof} for more details on Borel distribution and references therein). Further tree number $i $ is drawn uniformly among the trees of size $X_{i}$.
 
A classical quantity of interest is the number of trees of size $K$ in the forest (see, e.g., \cite{Kolchin84,Pavlov77,Pavlov96}). It means that we choose $Y_i = \indic_{\{X_i=K\}}$. Let us now assume that we condition on $\sum_{i=1}^NX_i=m$ with $m=k_n \rightarrow + \infty$, $N=N_n \rightarrow + \infty$. The framework is similar to the one of Subsection \ref{exocc} and we proceed analogously.  Assume $k_n/N_n \rightarrow \lambda$ and take $X_i^{(n)}$ having Borel distribution with parameter $\lambda_n=k_n/N_n$. 

\subsubsection{Bose-Einstein statistics}
This example is borrowed from \cite{Holst79}. Consider $N$ urns. Put $n$ indistinguishable balls in the urns in such a way that each distinguishable outcome has the same probability 
$$
1/ \begin{pmatrix}n+N-1\\n\end{pmatrix},
$$
see for example \cite{Feller68}. Let $Z_k$ be the number of balls in the
$k^{\textrm{th}}$ urn. It is well known that $(Z_1,\ldots,Z_N )$ is distributed as $\big(
X_1,\cdots,X_N\big)$ conditioned on $\sum_{i=1}^N X_i=n$, where $
X_1,\cdots,X_N$ are i.i.d.\ and geometrically distributed. 

%

\subsubsection{Hashing with linear probing}\label{subsubsec:hash}

Hashing with linear probing can be regarded as throwing $n$ balls
sequentially into $m$ urns at random; the urns are arranged in a circl and labeled. A ball that lands in an occupied urn is moved to the next empty urn, always
moving in a fixed direction. The length of the move is called the displacement of the ball, and we are interested in the sum $d_{m,n}$ of all displacements. We assume $n<m$ and denote $N=m-n$.

Janson \cite{Janson01a} proved that the length of the blocks (counting the
empty urn) and the sum of displacements inside each block are distributed as
$(X_1,Y_1),\ldots,(X_N,Y_N)$ conditioned on $\sum_{i=1}^NX_i=m$, where $(X_i,Y_i)$ are i.i.d.\ copies of a pair $(X, Y)$ of random variables, $X$ having the Borel distribution with any parameter $\lambda \in \intervalleof{0}{e^{-1}}$ (see section \ref{subsec:hash_proof} for more details on Borel distribution and references therein), and $Y$ given $X=l$ is distributed as $d_{l,l-1}$. As in \ref{exocc}, we assume that $m = k_n \to \infty$ and $N = N_n \to \infty$ with $k_n/N_n \to a \in \intervallefo{1}{+\infty}$. So, $\lambda_n \defeq (n_n/m_n) \exp(-n_n/m_n) \in \intervallefo{0}{e^{-1}}$ and $\lambda_n \to (1-1/a) \exp(-1+1/a) \eqdef \lambda$. If $X^{(n)}$ has Borel distribution with parameter $\lambda_n$, Corollary \ref{cor:va_iid} yields the desired Berry-Esseen bound.

\subsection{Conditional large deviation result}

In \cite{TFC12}, the authors proved a classical large deviation principle for the conditional distribution $\mathcal{L}_n$ which applies to examples 2.3.1 to 2.3.4. Their result \cite[Theorem 2.1]{TFC12} is the analogue of the central limit theorem of Janson \cite{Janson01}. The proof relies on Gärtner-Ellis theorem which requires the existence of the Laplace transform in a neighborhood of the origin. In the context of hashing, however, the joint Laplace transform is only defined on $(-\infty,a)\times(-\infty,0)$ for some $a>0$ and \cite[Theorem 2.1]{TFC12} cannot be applied. Consequently one needs a specific result in the case when the Laplace transform is not defined.

\begin{theorem}\label{th:nagaev_weak_array_cond_mob_cor}
Suppose that:
\renewcommand{\theenumi}{\upshape\sffamily(H\thetheorem{}.\arabic{enumi})}
\renewcommand{\labelenumi}{\theenumi}
\begin{enumerate}
\addtolength{\itemindent}{2em}
\item \label{hyp:Xmt2_weak_array_cond_mob_cor} $\log(\ET{X\chap}) = o({N_n}^{1/2})$ where $\ET{X\chap} \defeq \Var\left(X\chap\right)^{1/2}$;
\item \label{hyp:Xmt3_weak_array_cond_mob_cor} $\RHO{X\chap} \defeq \Espe\left[\left|X\chap-\Espe\left[X\chap\right]\right|^3\right] = o\left( N_n^{1/2} \ET{X\chap}^3 \right)$ ;
\item \label{hyp:Xphi_weak_array_cond_mob_cor}
there exists $c > 0$ such that, for all $n \geqslant 1$ and $s \in \intervalleff{-\pi}{\pi}$,
\[
\abs{\Espe\left[ e^{isX^{(n)}} \right]} \leqslant 1 - c \ET{X\chap}^2 s^2 ;
\]
\item \label{hyp:kn_weak_array_cond_mob_cor} $k_n = N_n\Espe\left[ X\chap \right] + O(\ET{X\chap} N_n^{1/2})$;
\item \label{hyp:Ymt2_weak_array_cond_mob_cor}
$\Var\left(Y\chap\right) = o\left( N_n^{1/2} \right)$.
\item \label{hyp:Yqueue_weak_array_cond_mob_cor} the right tail of $Y^{(n)}$ satisfies: there exist $\alpha > 0$ and $\beta>0$ such that, for all $y > 0$,
\begin{equation}\label{hyp:Yqueue1_weak_array_cond_mob_cor}
\liminf_{n \to \infty} \frac{1}{\sqrt{{N_n}y}} \log\Prob(Y^{(n)} \geqslant {N_n}y) \geqslant -\beta
\end{equation}
and
\begin{equation}\label{hyp:Yqueue2_weak_array_cond_mob_cor}
\limsup_{n \to \infty} \sup_{u \geqslant \sqrt{{N_n}y}} \frac{1}{\sqrt{u}} \log\Prob(Y^{(n)} \geqslant u) \leqslant -\alpha.
\end{equation}
\end{enumerate}
\renewcommand{\theenumi}{\arabic{enumi}}
\renewcommand{\labelenumi}{\theenumi.}
Then, for all $y > 0$,
\begin{align*}
- \beta\sqrt{y} & \leqslant \liminf_{n \to \infty} \frac{1}{\sqrt{N_n}} \log \Prob( T_n - \Espe\left[ T_n | S_n = k_n \right] \geqslant N_ny | S_n = k_n ) \\
 & \leqslant \limsup_{n \to \infty} \frac{1}{\sqrt{N_n}} \log \Prob( T_n - \Espe\left[ T_n | S_n = k_n \right] \geqslant N_ny | S_n = k_n ) \leqslant - \alpha \sqrt{y}.
\end{align*}
\end{theorem}

\begin{remark}\leavevmode
\begin{enumerate}
\item Notice the different nature of the assumptions on the standard deviations $\ET{X\chap}$ and $\ET{Y\chap}$. 
\item The small shift allowed in assumption \ref{hyp:kn_weak_array_cond_mob_cor} is the same as the one in assumption \ref{ass:exp_X} of Theorem \ref{th:BE_cond_strong}. When the joint Laplace transform is defined in a neighborhood of the origin, one can use exponential changes of probability: a first one is based on the Laplace transform of $X^{(n)}$ and leads to reduce the conditioning to the mean $N_n\Espe\left[ X\chap \right]$ of $S_n$ whereas the second relies on the Laplace transform of $Y^{(n)}$ and removes the conditioning leading to the study of a pair of random variables (see \cite{TFC12}). The large deviation principle is then proved for a larger range of shifts in the conditioning. 
\end{enumerate}
\end{remark}

The result deeply relies on the following unconditioned one.
\begin{theorem}\label{th:nagaev_weak_array_mob_v2}
For all $n \geqslant 1$, let $z_n$ be a positive number. Suppose that $N_n \to +\infty$ and that:
\renewcommand{\theenumi}{\upshape\sffamily(H\thetheorem{}.\arabic{enumi})}
\renewcommand{\labelenumi}{\theenumi}
\begin{enumerate}
\addtolength{\itemindent}{2em}
\item \label{hyp:z_weak_array_mob_v2} $\liminf z_n/N_n > 0$;
\item \label{hyp:mt2_weak_array_mob_v2} $\Var(Y^{(n)}) = o\left( N_n^{1/2} \right)$;
\item \label{hyp:queue1_weak_array_mob_v2_2} the right tail of $Y^{(n)}$ satisfies: there exist $\alpha > 0$ and $\beta>0$ such that
\begin{equation}\label{hyp:queue1_weak_array_mob_v2}
\liminf_{n \to \infty} \frac{1}{\sqrt{z_n}} \log\Prob(Y^{(n)} \geqslant z_n) \geqslant -\beta
\end{equation}
and
\begin{equation}\label{hyp:queue2_weak_array_mob_v2}
\limsup_{n \to \infty} \sup_{u \geqslant \sqrt{z_n}} \frac{1}{\sqrt{u}} \log\Prob(Y^{(n)} \geqslant u) \leqslant -\alpha.
\end{equation}
\end{enumerate}
\renewcommand{\theenumi}{\arabic{enumi}}
\renewcommand{\labelenumi}{\theenumi.}
Then
\begin{align*}
-\beta
 & \leqslant \liminf_{n \to \infty} \frac{1}{\sqrt{z_n}}\log \Prob( T_n - N_n \Espe[Y^{(n)}]\geqslant z_n ) \\
 & \leqslant \limsup_{n \to \infty} \frac{1}{\sqrt{z_n}}\log \Prob( T_n - N_n \Espe[Y^{(n)}] \geqslant z_n ) \leqslant - \alpha.
\end{align*}
\end{theorem}
\renewcommand{\theenumi}{\arabic{enumi}}
\renewcommand{\labelenumi}{\theenumi.} 

\begin{remark}
Assumption \ref{hyp:z_weak_array_mob_v2} naturally implies that $z_n$ goes to infinity with $n$.
\end{remark}

\section{Application to hashing with linear probing}\label{sec:hash}

In this section we show that the example of hashing with linear probing briefly presented in section \ref{subsubsec:hash} satisfies the hypotheses of Theorem \ref{th:nagaev_weak_array_cond_mob_cor}. We begin with a precise description of the model.

\subsection{Complements on the model}

Hashing with linear probing is a classical model in theoretical computer science which has been studied from a mathematical point of view by several authors \cite{FPV98, Janson01a, Janson05, Chassaing02, Marckert01-1}. For more details on the model, we refer to \cite{FPV98, Janson01a, Janson05}. The model describes the following experiment. One throws $n$ balls sequentially into $m$ urns at random; the urns are arranged in a circle and numbered. A ball that lands in an occupied urn is moved to the next empty urn, always moving in a fixed direction. The length of the move is called the displacement of the ball and we are interested in the sum of all displacements which is a random variable noted $d_{m,n}$. We assume $n<m$ and define $N=m-n$.

In order to make things clear, let us give an example. Assume that $n=8$, $m=10$, and $(6,9,1,9,9,6,2,5)$ are the addresses where the balls land. This sequence of addresses is called a \emph{hash sequence} of length $m$ and size $n$. Let $d_i$ be the displacement of ball $i$, then $d_1=d_2=d_3=0$. The ball number $4$ should land in the $9^{\textrm{th}}$ urn which is occupied by the second ball; thus it moves one step ahead and lands in urn $10$ so that $d_4=1$. The $5^{\textrm{th}}$ ball should land in the $9^{\textrm{th}}$ urn. Since it is not possible (the urn being occupied by the second ball), it moves to the $10^{\textrm{th}}$ urn which is also occupied; it then moves to the first urn (also occupied) and finally to the second urn so that $d_5=3$. And so on: $d_6=1,\ d_7=1,\ d_8=0$. Here, the total displacement equals $1+3+1+1=6$. After throwing all balls, there are $N=m-n$ empty urns. These divide the occupied urns into blocks of consecutive urns. For convenience, we consider the empty urn following a block as belonging to this block. In our example, there are two blocks: the first one containing urns $9,10,1,2,3$ (occupied), and urn $4$ empty, and the second one containing urns $5,6,7$ (occupied), and urn $8$ empty.

Janson \cite{Janson01a} proved that the lengths of the blocks (counting the last empty urn) and the sum of displacements inside each block are distributed as $(X_1, Y_1), \ldots, (X_N, Y_N)$ conditioned on $\sum_{i=1}^N X_i = m$, where $(X_i, Y_i)$ are i.i.d.\ copies of a pair $(X,Y)$ of random variables, $X$ having the Borel distribution with any parameter $\lambda \in \intervalleof{0}{e^{-1}}$ (see section \ref{subsec:hash_proof} for more details on Borel distribution and references therein) and the conditional distribution of $Y$ given $X = l$ being the same as the distribution of $d_{l, l-1}$. So, $d_{m, n}$ is distributed as $\sum_{i=1}^N Y_i$ conditioned on $\sum_{i=1}^N X_i = m$. The following lemma presents already known results on the total displacement $d_{n+1,n}$ that will be useful in the proofs.

\begin{lemma}\label{lem:basic_hashing}\leavevmode
\begin{enumerate}
\item The number of hash sequences of length $n+1$ and size $n$ is $(n+1)^{n}$.
\item One clearly has $0\leqslant d_{n+1,n}\leqslant \frac{n(n-1)}{2}.$
\item For any $y\geqp 0$, the function defined from $\N$ to $[0,1]$ by $n\mapsto \Prob(d_{n+1,n}\geqp y)$ is an increasing function of $n$.
\item The total displacement of any hash sequence $(h_{1},\ldots,h_{n })$ is invariant with respect to any permutation of the $h_{i}'$s. More precisely for any permutation $\sigma$ of $\{1,\dots,n\}$, the total displacement associated  to the hash sequence $(h_{1},\ldots,h_{n })$ is the same as  the total displacement associated to  the hash sequence $(h_{\sigma(1)},\ldots,h_{\sigma(n) })$.
\end{enumerate}
\end{lemma}

\begin{proof}[Proof of Lemma \ref{lem:basic_hashing}] The first three points are obvious. Let us prove the last one. It is a consequence of \cite[Lemma 2.1]{Janson01a}.  For any hash sequence $\left(h_{1},\ldots,h_{n}\right)$ and for any $i=0,\ldots,n+1$, let us define
\[
Z_i \defeq \mathrm{Card}\{k \in \intervallentff{1}{n}, h_k = i\}
\]
and $\Sigma_i \defeq \sum_{k=1}^{i} Z_{j}$ (notice that $Z_0 = 0$ and $\Sigma_0 = 0$). It is obvious that the sequence $(\Sigma_{i})_{i=0,\ldots,n+1}$ does not depend on the order of the hash sequence $\left(h_{1},\ldots,h_{n}\right)$. Now, formula (2.1) in \cite[p.\ 442]{Janson01a} establishes that 
\[
d_{n+1,n}=\sum_{i=1}^{n+1} H_{i}-n
\]
where $H_{i}$, the number of items that make attempt to be inserted in cell $i$, is related to the sequence $(\Sigma_{i})_{i=0,\ldots,n+1}$ with the following formula (see \cite[Lemma 2.1]{Janson01a}):
\[
H_{i} = \Sigma_{i}-i-\min_{k<i}(\Sigma_{k}-k)+1.
\]
Hence $d_{n+1,n}$ does not depend on the order of the hash sequence $\left(h_{1},\ldots,h_{n}\right)$. 
\end{proof}

Using the results in \cite{FPV98,Janson01,Janson01a}, we can prove that the joint Laplace transform of $(X,Y)$ is  only defined on $(-\infty,a)\times(-\infty,0)$ for some positive $a$. Hence, Theorem 2.1 of \cite{TFC12} can not be applied here.

\subsection{Large deviations for hashing with linear probing}

In order to provide large deviation bounds for $d_{m, n}$, we need to describe the asymptotic behavior of $\Prob(Y \geqslant y)$, which is given in the following proposition.

\begin{proposition}\label{prop:queue_Y} Let $\lambda$ be the parameter of the Borel distribution of $X$ be such that $\kappa \defeq -\log(\lambda) - 1 \leqp \log(2)$. Then,
\begin{equation}\label{eq:queue_Y}
- \beta  \leqp \liminf\limits_{y\rightarrow +\infty} \frac{1}{\sqrt y}\log \Prob(Y \geqslant y) \leqp \limsup\limits_{y \rightarrow +\infty} \frac{1}{\sqrt y}\log \Prob(Y \geqslant y) \leqp - \alpha,
\end{equation}
with
\[
\alpha \defeq \kappa \sqrt{2} \qquad \text{and} \qquad
\beta \defeq 2\kappa \sqrt{\left(1+\frac{1}{\kappa}\right) \left(1+\frac{1+\log 2}{\kappa}\right)}.
\]
\end{proposition}

Now, for all $n \geqslant 1$, let $m_n$ and $n_n$ be integers such that $n_n < m_n$, and $N_n \defeq m_n - n_n$. Suppose that $m_n/N_n \to a \in \intervallefo{1}{+\infty}$. We introduce $\lambda_n \defeq (n_n/m_n) \exp(-n_n/m_n) \in \intervallefo{0}{e^{-1}}$. Hence $\lambda_n \to (1-1/a) \exp(-1+1/a) \eqdef \lambda$. To apply Proposition \ref{prop:queue_Y}, suppose that $\lambda \geqslant (2e)^{-1}$. Let $(X_i^{(n)}, Y_i^{(n)})_{i=1, 2, \ldots, N_n}$ be i.i.d.\ copies of $(X\chap , Y\chap)$, $X\chap$ following Borel distribution with parameter $\lambda_n$ (so that $\Espe[X\chap] = m_n/N_n$), and $Y\chap$ given $X\chap = l$ being distributed as $d_{l, l-1}$. Let
\[
S_n \defeq \sum_{i=1}^{N_n} X_i^{(n)} \quad \text{and} \quad  T_n \defeq \sum_{i=1}^{N_n} Y_i^{(n)}.
\]
The total displacement $d_{m_n, n_n}$ is distributed as the conditional distribution of $T_n$ given $S_n = m_n$. Since assumptions \ref{hyp:Xmt2_weak_array_cond_mob_cor} to \ref{hyp:Ymt2_weak_array_cond_mob_cor} are also satisfied by $\left(X_i^{(n)}, Y_i^{(n)}\right)$ ($i=1, 2, \ldots, N_n$), we can apply Theorem \ref{th:nagaev_weak_array_cond_mob_cor}.


\begin{proposition}[Large deviations for hashing with linear probing] \label{prop:hash_cond}
For $\alpha$ and $\beta$ defined in Proposition \ref{prop:queue_Y} and $k_n = m_n$, assumptions \ref{hyp:Xmt2_weak_array_cond_mob_cor} to \ref{hyp:Yqueue_weak_array_cond_mob_cor} are satisfied. Then, for all $y > 0$,
\begin{align*}
- \beta\sqrt{y} & \leqslant \liminf_{n \to \infty} \frac{1}{\sqrt{N_n}} \log \Prob(d_{m_n, n_n} - \Espe[d_{m_n, n_n}] \geqslant N_n y) \\
 & \leqslant \limsup_{n \to \infty} \frac{1}{\sqrt{N_n}} \log \Prob(d_{m_n, n_n} - \Espe[d_{m_n, n_n}] \geqslant N_n y) \leqslant - \alpha \sqrt{y}.
\end{align*}
\end{proposition}

\subsection{Proof of Proposition \ref{prop:queue_Y}} \label{subsec:hash_proof}

We start computing the asymptotic tail behavior of $X$. Remind that $X$ has Borel distribution with parameter $\lambda \in \intervalleof{0}{e^{-1}}$ which means that
\[
\Prob(X=n)=\frac{1}{T(\lambda)} \frac{\lambda^n n^{n-1}}{n!},
\]
where $T$ is the well-known tree function (see, e.g., \cite{FPV98} or \cite{Janson01a} for more details). We define $\kappa \in ]0,+\infty[$ by $\kappa \defeq -\log(\lambda) - 1$.


\begin{lemma}\label{lem:queue_X}\leavevmode
\begin{enumerate}
\item[(i)] The asymptotic behavior of $X$ is given by
\begin{equation}\label{eq:queue_X}
\log \Prob(X=n) = -\kappa n (1+o(1)).
\end{equation}
\item[(ii)] The asymptotic tail behavior of $X$ is given by
\begin{equation}\label{eq:queue_X_2}
\log \Prob(X\geqslant n) = -\kappa n(1+o(1)).
\end{equation}
\end{enumerate}
\end{lemma}

\begin{proof}
(i) By Stirling formula,
\[
\log \Prob(X=n) = \log \left(\frac{1}{\sqrt{2\pi}T(\lambda)} \frac{(\lambda e)^{n}}{n^{3/2}}\right)(1+o(1))= -\kappa n(1+o(1)).
\]

(ii) Similarly, using Stirling formula,
\begin{align*}
\Prob(X\geqslant n)
 & = \sum_{k\geqslant n}\Prob(X=k)=\frac{1}{\sqrt{2\pi}T(\lambda)} \sum_{k\geqslant n}e^{-\kappa{k}(1+o(k))}k^{-3/2} \\
 & = \frac{1}{\sqrt{2\pi}T(\lambda)} \sum_{k\geqslant n}e^{-\kappa{k}(1+o(k))}.
\end{align*}

Let $\varepsilon>0$. Then there exists $n_0\in \N$ such that, for any $k\geqslant n_0$, $|o(k)|\leqslant \varepsilon$. Thus, for any $n\geqslant n_0$,
\[
\sum_{k\geqslant n}e^{-\kappa{k}(1+\varepsilon)} \leqslant  \sqrt{2\pi}T(\lambda) \Prob(X\geqslant n)\leqslant \sum_{k\geqslant n}e^{-\kappa{k}(1-\varepsilon)}.
\]
Using the fact that $\lambda e<1$, we get
\begin{align*}
\log\left(\frac{1}{\sqrt{2\pi}T(\lambda)} \sum_{k\geqslant n}e^{-\kappa{k}(1\pm\varepsilon)}\right)&=\log\left(\frac{e^{-\kappa n}}{\sqrt{2\pi}T(\lambda)} \frac{e^{\pm\kappa n \varepsilon}}{1-e^{-\kappa (1\pm\varepsilon)}}\right)\\
&= -\kappa n(1\pm \varepsilon)(1+o(1)),
\end{align*}
which leads to the required result when $\varepsilon$ goes to $0$.
\end{proof}

\begin{proof}[Proof of the upper bound in \eqref{eq:queue_Y}]
Let $y>0$ and $n_y$ be the ceiling of the positive solution of $2y=n(n-1)$:
\begin{equation}\label{eq:def_ny}
n_y=\left\lceil \sqrt{2y+\frac 14} +\frac 12 \right\rceil.
\end{equation} 

Since $Y$ conditionally to $X=n+1$ is distributed as $d_{n+1,n}$, we get
\[
\Prob(Y\geqslant y)= \sum_{n=n_y}^{+\infty} \Prob(d_{n+1,n}\geqslant y)\Prob(X=n+1)\leqslant \sum_{n=n_y}^{+\infty}\Prob(X=n+1)= \Prob(X\geqslant n_y).
\]
By \eqref{eq:queue_X_2} and the fact that $n_y=\sqrt{2y}(1+o(1))$, we finally conclude that 
\[
\limsup\limits_{y\rightarrow +\infty} \log \Prob(Y\geqp y)\leqp -\kappa \sqrt{2 y}.
\]
\end{proof}

\begin{proof}[Proof of the lower bound in \eqref{eq:queue_Y}]
Let $y>0$. For any $m_y\in \N^*$ such that $m_y\geqp n_y$, one has
\begin{align*}
\Prob(Y\geqslant y) &=\sum_{n=n_y}^{+\infty} \Prob(d_{n+1,n}\geqslant y)\Prob(X=n+1)\\
&\geqslant \Prob\left(d_{m_y+1,m_y}\geqslant y\right)\Prob(X=m_y+1)\\
\end{align*}

So, we are interested in the hash sequences of length $m_y+1$ and size $m_y$ that realize a total displacement greater than $y$. More precisely, we want to evaluate the probability $\Prob\left(d_{m_y+1,m_y}\geqp y\right)$ or at least to bound it from below. In that view, for any $0\leq k\leqp \frac{m_{y}}{2} $ consider the following hash sequence:
\begin{equation}\label{ex:arrivals}
\left(1,\, 1,\, 2,\, 2,\,\ldots\,  k,\, k,\, k+1,\, k+2,\, \ldots,\, m_y-k\right).
\end{equation}
On the one hand, it is decomposed into $m_y-2k$ single numbers and $k$ pairs leading to a hash sequence of size $m_y$ as required. On the other hand, each pair $(q,\, q)$ ($q=1\ldots k$) realizes a displacement equal to $(q-1)+q$ while each singleton $q$ ($q=k+1 \ldots m_y-k$) realizes a displacement equal to $k$. The total displacement is then $k(m_y-k)$. It remains to 
choose $m_{y}$ and  $0\leqp k\leqp\frac{m_{y}}{2}  $ such that $k(m_y-k)\geqp y$ in order to obtain the best possible lower bound. 
 
Moreover as mentioned in Lemma \ref{lem:basic_hashing} the  total displacement associated to any hash sequence does not depend on the order of the hash sequence.  One can consider all the permutations of the hash sequence defined in \eqref{ex:arrivals} whose total number is given by
\[
\binom{m_y}{1} \binom{m_y-1}{1} \ldots \binom{2k+1}{1} \binom{2k}{2} \binom{2k-2}{2} \ldots \binom{2}{2}=\frac{m_y!}{2^{k}}.
\]
%
As a consequence, $\Prob(Y\geqslant y)$ is bounded from below by $\frac{1}{(m_y+1)^{m_y}}\frac{m_y!}{2^{k}}\Prob(X=m_y+1).$ By Stirling formula, $n!\underset{n}{\sim}\sqrt{2\pi n}\left(\frac{n}{e}\right)^{n}$ and the asymptotic behavior of $X$ given in \eqref{eq:queue_X},
\begin{align}\label{eq:kappa}
\log \left(\frac{1}{(m_y+1)^{m_y}}\frac{m_y!}{2^{k}}\,\Prob(X=m_y+1)\right)
&\underset{y}{\sim} -(\kappa+1) m_y-k\log 2. 
\end{align} 
Now the inequality  $k(m_{y}-k)\geqp y$ admits solutions as soon as $m_{y}\geqp2\sqrt{y}$. Hence we take $m_{y}=2t\sqrt{y}$ for some $t\geqp 1$.  Simple computation shows that the best possible choices for $k$ and  $t$ are $k=\frac{m_{y}-\sqrt{m_{y}^{2}-4y}}{2}$ and $t=\left(1+2\frac{\kappa+1}{\log2}\right)\left(\left(1+2\frac{\kappa+1}{\log2}\right)^{2}-1\right)^{-1/2}$.
Plugging the values of $m_y$ and $k$ into \eqref{eq:kappa} leads to the value
\[
-2\kappa \sqrt{\left(1+\frac{1}{\kappa}\right)\left(1+\frac{1+\log 2}{\kappa}\right)}\sqrt y;
\]
which completes the proof of the minoration. 
\end{proof}

\section{Proofs}\label{sec:preuve}

\subsection{Notations and technical results}

The proofs of Theorems \ref{th:BE_cond_strong} and \ref{th:nagaev_weak_array_cond_mob_cor} intensively rely on the use of Fourier transforms. Define $\varphi_n$ and $\psi_n$ by
\begin{align}
\varphi_n(s,t) &\defeq \Espe\left[\exp\left\{is\left(X\chap-\Espe\left[X\chap\right]\right)+it \left(Y\chap-\Espe\left[Y\chap\right]\right)\right\}\right]\label{def:varphi_n}\\
\text{and} \quad \psi_n(t)&\defeq2\pi \Prob(S_n=k_n) \Espe\left[\exp\left\{it\left(U_n-N_n\Espe\left[Y\chap\right]\right)\right\}\right].\label{def:psi_n}
\end{align}
In this first section, we establish some properties of those two functions. First notice that we have $\varphi_n(s,0) = e^{-is\Espe\left[X\chap\right]} \Espe\left[e^{is X\chap}\right]$ and $\psi_n(0)= 2\pi \Prob(S_n=k_n)$.

\begin{lemma} \label{lem:bartlett}
One has
\begin{equation}\label{eq:bartlett}
\psi_n(t)=\frac{1}{\sigma_{X\chap} N_n^{1/2}} \int_{-\pi \sigma_{X\chap} N_n^{1/2}}^{\pi \sigma_{X\chap} N_n^{1/2}} e^{-is\bnx\left(k_n-N_n\Espe\left[X\chap\right]\right)} \varphi_n^{N_n}\left(\frac{s}{\bnxm},t\right) ds.
\end{equation}
\end{lemma}

\begin{proof}
Since
\[
\int_{-\pi}^\pi e^{is(S_n - k_n)} ds = 2 \pi \indic_{\{ S_n = k_n \}},
\]
we have
\begin{align*}
\psi_n(t) & = 2\pi \Prob(S_n=k_n) \Espe\left[\exp\left\{it\left(U_n-N_n\Espe\left[Y\chap\right]\right)\right\}\right] \\
 & = 2\pi \Espe\left[ \exp\left\{it  \left(T_n -N_n\Espe\left[Y\chap\right]\right)\right\} \indic_{S_n = k_n} \right] \\
 & = \int_{-\pi}^{\pi} \Espe\left[\exp\left\{is \left(S_n -k_n\right)+it  \left(T_n-N_n\Espe\left[Y\chap\right]\right) \right\} \right]ds \\
 & = \int_{-\pi}^{\pi} e^{-is\left(k_n-N_n\Espe\left[X\chap\right]\right)}\varphi_n^{N_n}(s,t)ds,
\end{align*}
which leads to the result after the change of variable $s'=s\sigma_{X\chap} N_n^{1/2}$.
\end{proof}

\begin{lemma} \label{lem:maj_expo} \leavevmode
\begin{enumerate}
\item[(i)] Under assumption \ref{ass:fc_X}, for any integer $l \geqslant 0$, and for $\abs{s} \leqslant \pi \ET{X\chap} N_n^{1/2}$, $\abs{t}  \leqslant \eta_0 \ET{Y\chap} N_n^{1/2}$,
\begin{equation} \label{maj_expo_st}
\abs{\varphi_n^{N_n-l} \bigg(\frac{s}{\bnxm}, \frac{t}{\bnym} \bigg)} \leqslant e^{-(s^2+t^2) \cdot c_5 (N_n - l)/N_n}.
\end{equation}
\item[(ii)] Under assumption \ref{hyp:Xphi_weak_array_cond_mob_cor}, for any integer $l \geqslant 0$, and for $\abs{s} \leqslant \pi \ET{X\chap} N_n^{1/2}$,
\begin{equation} \label{maj_expo_s}
\abs{\varphi_n^{N_n-l} \bigg(\frac{s}{\bnxm}, 0 \bigg)} \leqslant e^{-s^2 \cdot c (N_n - l)/N_n}.
\end{equation}
\end{enumerate}
\end{lemma}

\begin{proof}
The proof is a mere consequence of the inequality $1 + x \leqslant e^x$.
\end{proof}

In the sequel, we also need different controls on the first derivative of $\varphi_n$ with respect to the first variable.
\begin{lemma} \label{lem:maj_der_phi_n}
For any $s$ and $t$, one has:
\begin{enumerate}
\item[(i)] \begin{equation} \label{eq:maj_der_phi_n}
\left|\frac{\partial \varphi_n}{\partial t} \left(\frac{s}{\bnxm},\frac{t}{\bnym}\right)\right| \leqp \frac{\ET{Y\chap}}{N_n^{1/2}} (|s|+|t|) ;
\end{equation}
\item[(ii)] \begin{align}\label{esp1b_array}
\left\lvert \frac{\partial \varphi_n}{\partial t} \left(\frac{s}{\ET{X^{(n)}} N_n^{1/2}}, \right.\right. & \left.\left.  \frac{t}{\ET{Y^{(n)}} N_n^{1/2}} \right)\right\rvert \\
 & \leqslant \frac{\ET{Y\chap}}{N_n^{1/2}} (|s|r_n + |t|) + \frac{\ET{Y\chap}}{N_n} \bigg[ \frac{s^2}{2} \bigg( \frac{\RHO{X\chap}}{\ET{X\chap}^3}\bigg)^{2/3} \bigg( \frac{\RHO{Y\chap}}{\ET{Y\chap}^3}\bigg)^{1/3} \nonumber \\
 & \qquad + \abs{st} \bigg( \frac{\RHO{X\chap}}{\ET{X\chap}^3}\bigg)^{1/3} \bigg( \frac{\RHO{Y\chap}}{\ET{Y\chap}^3}\bigg)^{2/3} + \frac{t^2}{2} \bigg( \frac{\RHO{Y\chap}}{\ET{Y\chap}^3}\bigg) \bigg].
\end{align}
\end{enumerate}
\end{lemma}

\begin{proof}
We apply Taylor Theorem to the function defined by
\[
(s,t) \mapsto f(s,t)=\frac{\partial \varphi_n}{\partial t} \left(\frac{s}{\bnxm},\frac{t}{\bnym}\right).
\]
We conclude to (i) using
\[
\left|f(s,t)-f(0,0)\right|\leqp |s|\sup_{\theta, \theta' \in [0,1]} \left|\frac{\partial f}{\partial s}\left(\theta s,\theta' t\right)\right|+|t|\sup_{\theta, \theta' \in [0,1]} \left|\frac{\partial f}{\partial t}\left(\theta s,\theta' t\right)\right|
\]
and to (ii) using
\begin{align*}
\left|f(s,t)-f(0,0)\right|
 & \leqp |s| \left|\frac{\partial f}{\partial s}\left(0,0\right)\right|+|t| \left|\frac{\partial f}{\partial t}\left(0,0\right)\right| + \frac{s^2}{2}\sup_{\theta, \theta' \in [0,1]} \left|\frac{\partial^2 f}{\partial^2 s}\left(\theta s,\theta' t\right)\right|\\
 & \qquad + |st|\sup_{\theta, \theta' \in [0,1]} \left|\frac{\partial^2 f}{\partial t\partial s}\left(\theta s,\theta' t\right)\right|+\frac{t^2}{2}\sup_{\theta, \theta' \in [0,1]} \left|\frac{\partial^2 f}{\partial^2 t}\left(\theta s,\theta' t\right)\right|
\end{align*}
\end{proof}

\begin{proposition} \label{sigmaXN} \leavevmode
\begin{enumerate}
\item Under assumption \ref{ass:rho_X}, one has $\ET{X\chap} \geqslant (4 c_2^3)^{-1}$.
\item Under assumption \ref{hyp:Xmt3_weak_array_cond_mob_cor}, one has $\ET{X^{(n)}} N_n^{1/2} \to + \infty$.
\end{enumerate}
\end{proposition}

\begin{proof}
The proofs of both results rely on the fact that, for any integer-valued random variable $X$ (see \cite[Lemma 4.1.]{Janson01}),
\[
\ET{X}^2 \leqslant 4 \Espe\left[\abs{X - \Espe\left[X\right]}^3\right].
\]
The conclusion follows, using hypothesis \ref{ass:rho_X} (resp. \ref{hyp:Xmt3_weak_array_cond_mob_cor}).
\end{proof}

\begin{proposition}\label{tll_weak_array_cond_mob_cor}
We assume hypotheses \ref{ass:rho_X}, \ref{ass:fc_X}, and \ref{ass:exp_X} (or \ref{hyp:Xmt3_weak_array_cond_mob_cor}, \ref{hyp:Xphi_weak_array_cond_mob_cor} and \ref{hyp:kn_weak_array_cond_mob_cor}). Then there exists $m > 0$ such that
\[
\Prob(S_n = k_n) \geqslant \frac{m}{2\pi \ET{X^{(n)}} N_n^{1/2}}.
\]
\end{proposition}

\begin{proof}
Only consider the indices $n$ for which $\ET{X\chap} < +\infty$. Remember that $\varphi_n(s,0) = \Espe\left[ e^{is(X^{(n)} - \Espe[X^{(n)}])} \right]$ and
\[
\psi_n(0) = 2\pi \Prob(S_n = k_n) = \frac{1}{\ET{X^{(n)}} N_n^{1/2}} \int_{-\pi \ET{X^{(n)}} N_n^{1/2}}^{\pi \ET{X^{(n)}} N_n^{1/2}} e^{-i s v_n} \varphi_n^{N_n}\left( \frac{s}{\ET{X^{(n)}} N_n^{1/2}},0 \right) ds
\]
where $v_n = \frac{k_n - {N_n}\Espe\left[X^{(n)}\right]}{\ET{X^{(n)}} N_n^{1/2}}$, by lemma \ref{lem:bartlett}. Let us prove that the sequence
\[
(u_n)_n = \left( \psi_n(0) \ET{X^{(n)}} N_n^{1/2} e^{v_n^2/2} \right)
\]
converges to $\sqrt{2\pi}$, from which the conclusion follows, since $(v_n)_n$ is bounded by \ref{ass:exp_X} (or \ref{hyp:kn_weak_array_cond_mob_cor}) and $\Prob(S_n = k_n) > 0$ for all $n$. Inequality \eqref{maj_expo_st} with $l = 0$ and $t = 0$ (or \eqref{maj_expo_s} with $l = 0$) implies that the sequence $(u_n)_n$ is bounded. Let us prove that $\sqrt{2\pi}$ is the only accumulation point of $(u_n)_n$. Let $\phi(n)$ such that $(u_{\phi(n)})_n$ converges. Even if it means extracting more, we can suppose that $(v_{\phi(n)})_n$ converges. Let $v = \lim v_{\phi(n)}$. Using Taylor Theorem, there exists $t \in \R$ such that
\begin{align*}
\left|\varphi_n \left(\frac{s}{\ET{X^{(n)}} N_n^{1/2}},0\right) - 1 + \frac{s^2}{2 {N_n}}\right| &\leqp \frac{\abs{s}^3}{6 \ET{X^{(n)}}^3 N_n^{3/2}} \Espe\left[ \left|X^{(n)} - \Espe\left[X^{(n)}\right]\right|^3\right] = o\left( \frac{1}{{N_n}} \right)
\end{align*}
where the last equality follows from hypothesis \ref{ass:rho_X} (or \ref{hyp:Xmt3_weak_array_cond_mob_cor}). Now,
\[
e^{-i s v_{\phi(n)}} \varphi_{\phi(n)}^{N_{\phi(n)}}\left( \frac{s}{\sigma_{X^{(\phi(n))}} \sqrt{N_{\phi(n)}}},0 \right) \to e^{-i s v - s^2/2} = e^{-v^2/2} e^{-(s+iv)^2/2}
\]
and, by Lebesgue dominated convergence theorem and the fact that $\ET{X\chap} N_n^{1/2} \to +\infty$ (see Proposition \ref{sigmaXN}),
\[
\psi_{\phi(n)}(0) \sigma_{X^{(\phi(n))}} \sqrt{N_{\phi(n)}} e^{v_{\phi(n)}^2/2} \to \sqrt{2\pi}.
\]
\end{proof}

\subsection{Proof of Theorem \ref{th:BE_cond_strong}}

Part a) is Proposition \ref{tll_weak_array_cond_mob_cor} with $\tilde{c}_5 = m$. Now we follow the procedure of Janson \cite{Janson01} to uncorrelate $X\chap$ and $Y\chap$ and center the variable $Y\chap$. We replace $Y^{(n)}$ by the projection
\begin{align*}
Y^{'(n)} \defeq Y\chap -\Espe[Y\chap]-\frac{\Cov(X\chap,Y\chap)}{\sigma_{X\chap}^2}\left(X\chap-\Espe[X\chap]\right).
\end{align*}
Then $\Espe[Y^{'(n)}] = 0$ and $\Cov(X\chap,Y^{'(n)}) = \Espe[X\chap Y^{'(n)}] = 0$. Besides, assumptions \ref{ass:fc_X} and \ref{ass:corr} are verified by $Y^{'(n)}$. By assumption \ref{ass:corr},
\[
\sigma_{Y^{'(n)}}^2 = \sigma_{Y\chap}^2 (1 - r_n^2) \in [\tilde{c}_3^2(1 - c_6^2), c_3^2],
\]
so \ref{ass:var_Y} is satisfied by $Y^{'(n)}$. Finally, by Minkowski Inequality, assumptions \ref{ass:rho_X} and \ref{ass:rho_Y}, and the fact that $\abs{r_n} \leqslant 1$,
\begin{align*}
\norme{Y^{'(n)}}_3 & \leqp \norme{Y^{(n)}-\Espe[Y\chap]}_3 + \frac{\abs{r_n} \sigma_{X\chap}\sigma_{Y\chap}}{\sigma_{X\chap}^2} \norme{X^{(n)}-\Espe[X\chap]}_3 \\
 & \leqp \rho_{Y\chap}^{1/3} + r_n\sigma_{Y\chap} \frac{\rho_{X\chap}^{1/3}}{\sigma_{X\chap}} \\
  & \leqslant \ET{Y\chap}(c_2 + c_4).
\end{align*}
Hence $Y^{'(n)}$ satisfies assumption \ref{ass:rho_Y}. Consequently, all conditions hold for the pair $(X\chap,Y^{'(n)})$ too. Finally,
\[
T'_n \defeq \sum_{i=1}^{N_n} Y^{'(n)}_i=T_n-N_n\Espe\left[Y\chap\right]-\frac{\Cov(X\chap,Y\chap)}{\sigma_{X\chap}^2}\left(S_n-N_n \Espe\left[X\chap\right]\right).
\]
So, conditioned on $S_n=k_n$, we have $T'_n=T_n-N_n\Espe\left[Y\chap\right]-r_n\frac{\ET{Y\chap}} {\ET{X\chap}}(k_n - N_n\Espe[X\chap])$. Hence the conclusions for $\left(X\chap,Y\chap\right)$ and $\left(X\chap,Y^{'(n)}\right)$ are the same. Thus, it suffices to prove the theorem for $\left(X\chap,Y^{'(n)}\right)$; in other words, we may henceforth assume that $\Espe\left[Y\chap\right]=\Espe\left[X\chap Y\chap\right]=0$. Note that in that case $\tau_n^2=\sigma_{Y\chap}^2$.

\begin{proof}[Proof of Theorem \ref{th:BE_cond_strong} - Part b)]
We follow the classical proof of Berry-Esseen (see e.g. \cite{Feller71}) combined with the procedure of Quine and Robinson \cite{QR82} to establish the result of Theorem \ref{th:BE_cond_strong}.

As shown in Loève \cite{Loeve55} (page 285) or Feller \cite{Feller71}, the left hand side of \eqref{eq:be} is dominated by 
\begin{equation}\label{eq:loeve}
\frac{2}{\pi}\int_{0}^{\eta \bnym} \left|\frac{\psi_n(u/\bnym)}{2\pi\Prob(S_n=k_n)}-e^{-u^2/2}\right|\frac{du}{u}+\frac{24 \bny}{\eta \pi \sqrt{2\pi}}
\end{equation}
where $\eta > 0$ will be specified later. From Lemma \ref{lem:bartlett} and a Taylor expansion,
\begin{align*}
&u^{-1}\left|\frac{\psi_n(u/\bnym)}{2\pi\Prob(S_n=k_n)}-e^{-u^2/2}\right| = u^{-1}e^{-u^2/2}\left|\frac{e^{u^2/2}\psi_n(u/\bnym)}{2\pi\Prob(S_n=k_n)}-1\right| \\
& \leqp  e^{-u^2/2}  \sup_{0\leqp \theta \leqp u} \left|\frac{\partial}{\partial t}\left[\frac{e^{t^2/2}\psi_n(t/\bnym)}{2\pi\Prob(S_n=k_n)}\right]\right|_{t=\theta} \\
& \leqp  c_n^{-1}e^{-u^2/2} \sup_{0\leqp \theta \leqp u} \left\{\int_{-\pi \bnxm}^{\pi \bnxm}   \left|\frac{\partial}{\partial t}\left[e^{t^2/2}\varphi_n^{N_n}\left(\frac{s}{\bnxm},\frac{t}{\bnym}\right)\right]\right|_{t=\theta} ds\right\}
\end{align*}
where $c_n \defeq 2\pi\Prob(S_n=k_n) \bnxm \geqslant \tilde{c}_5$ and $v_n=\frac{k_n - {N_n}\Espe\left[X^{(n)}\right]}{\ET{X^{(n)}} N_n^{1/2}}$ has already been defined in the proof of Proposition \ref{tll_weak_array_cond_mob_cor}. Now we split the integration domain of $s$ into 
\[
A_1 \defeq \left\{s:\; |s|< \varepsilon \bnxm\right\} \quad \textrm{and} \quad A_2 \defeq \left\{s:\; \varepsilon \bnxm\leqp |s| \leqp \pi \bnxm\right\},
\]
(where $0 < \varepsilon < \pi$ will be specified later) and decompose 
\begin{equation}\label{eq:psi_I}
u^{-1}\left|\frac{\psi_n(u/\bnym)}{2\pi\Prob(S_n=k_n)}-e^{-u^2/2}\right|\leqp \sup_{0 \leqp \theta\leqp u} \left[I_1(u, \theta) + I_2(u, \theta)\right],
\end{equation}
where
\begin{align}
I_1(u, \theta)&= c_n^{-1}  \int_{A_1}  e^{-(u^2+s^2)/2} \abs{\left(\frac{\partial}{\partial t} \left[e^{(t^2+s^2)/2}\varphi_n^{N_n}\left(\frac{s}{\bnxm},\frac{t}{\bnym}\right)\right]\right)_{t=\theta}} ds, \label{def:I_1}\\
I_2(u, \theta)&= c_n^{-1}  e^{-u^2/2} \int_{A_2} \abs{\left(\frac{\partial}{\partial t}\left[e^{t^2/2}\varphi_n^{N_n}\left(\frac{s}{\bnxm},\frac{t}{\bnym}\right)\right]\right)_{t=\theta}} ds. \label{def:I_2}
\end{align}
To bound $I_1(u, \theta)$, we use a result due to Quine and Robinson (\cite[Lemma 2]{QR82}).

\begin{lemma}\label{lem:lem2}[Lemma 2 in \cite{QR82}]
Define
\[
l_{1,n} \defeq \rho_{X\chap} \sigma_{X\chap}^{-3} N_n^{-1/2} \qquad \text{and} \qquad l_{2,n} \defeq \rho_{Y\chap} \sigma_{Y\chap}^{-3} N_n^{-1/2}.
\]
If $l_{1,n} \leqslant 1$ and $l_{2,n} \leqslant 1$, then, for all
\[
(s, t) \in R \defeq \left\{(s,t):\; |s|<\frac{2}{9}l_{1,n}^{-1}, |t|<\frac{2}{9}l_{2,n}^{-1}\right\},
\]
we have
\begin{align}\label{eq:lem2}
\left\lvert \frac{\partial}{\partial t}\left[e^{(s^2+t^2)/2} \right.\right.
 & \left.\left. \varphi_n^{N_n}\left(\frac{s}{\bnxm},\frac{t}{\bnym}\right)\right]\right\rvert \nonumber \\
 & \leqp C_0(|s|+|t|+1)^3(l_{1,n}+l_{2,n})\exp\left\{\frac{11}{24}\left(s^2+t^2\right)\right\}
\end{align}
with
\[
C_0 \defeq 98.
\]
\end{lemma}

\begin{proof}
We refer to the proof in the appendix of \cite{QR82}. The condition $l_{1,n} < 12^{-3/2}$ and $l_{2,n} < 12^{-3/2}$ appearing in \cite[Lemma 2]{QR82} can be replaced by $l_{1,n} \leqslant (33/32)^{3/2}$ and $l_{2,n} \leqslant (33/32)^{3/2}$ since the factor $8/27$ in (A4) of their proof can be replaced by a factor $1/27$. Since we do not provide the best constants here, we simply suppose $l_{1,n} \leqslant 1$ and $l_{2,n} \leqslant 1$. Finally, $C_0$ has to be greater than $4$ and
\begin{align*}
\sup_{(v, s) \in \R^2}
 & \frac{27(\abs{v} + 2\abs{s})(\abs{v}^3 + \abs{s}^3)}{(\abs{v} + \abs{s} + 1)^3} e^{-(v^2 + s^2)/24} \\
 & \leqslant 54 \cdot (\abs{v} + \abs{s}) e^{-(v^2 + s^2)/24} \\
 & \leqslant 108 \cdot \sqrt{6} \sqrt{\frac{v^2+s^2}{12}} e^{-(v^2 + s^2)/24} \leqslant \frac{108 \cdot \sqrt{6}}{e} \leqslant 98.
\end{align*}
\end{proof}

By assumptions \ref{ass:rho_X} and \ref{ass:var_X},
\begin{equation}\label{eq:l1n_cond}
l_{1,n} \leqslant c_2^3 N_n^{-1/2} \leqslant c_2^3 c_1 \sigma_{X\chap}^{-1} N_n^{-1/2},
\end{equation}
which implies that $\bnxm \leqp c_2^{-3} c_1^{-1} l_{1,n}^{-1}$.
Similarly,
\begin{equation}\label{eq:l2n_cond}
l_{2,n} \leqslant c_4^3 N_n^{-1/2} \leqslant c_4^3 c_3 \sigma_{Y\chap}^{-1} N_n^{-1/2},
\end{equation}
and $\bnym \leqp c_4^{-3} c_3^{-1} l_{2,n}^{-1}$.
Assume henceforth that
\begin{equation}\label{eq:epsilon_eta}
\varepsilon \defeq \min \bigg( \frac{2}{9} c_1 c_2^3, \pi \bigg)  \quad \textrm{and} \quad \eta \defeq \min \bigg( \frac{2}{9} c_3 c_4^3, \eta_0 \bigg).
\end{equation}

\begin{lemma}\label{lem:I1}
There exists a positive constant $C_1$ such that 
\begin{equation}\label{eq:I1}
\int_0^{\eta \bnym} \sup_{0\leqp \theta \leqp u} I_1(u, \theta) du \leqp \frac{C_1}{N_n^{1/2}}.
\end{equation}
\end{lemma}

\begin{proof}
Conditions \eqref{eq:epsilon_eta} imply that, on $A_1$, 
\begin{align*}
|s| & < \varepsilon \bnxm \leqslant \frac{2}{9} l_{1,n}^{-1} \\ 
\text{and} \quad |\theta| & \leqslant |u| \leqslant \eta \bnym \leqslant \frac{2}{9} l_{2,n}^{-1},
\end{align*}
which ensures that $(s, u) \in R$ as specified in Lemma \ref{lem:lem2}. Moreover, since we have $N_n \geqslant \max(c_2^6, c_4^6)$ (cf.\ hypothesis in \ref{th:BEb}), $l_{1,n} \leqslant 1$ and $l_{2,n} \leqslant 1$. Now applying Lemma \ref{lem:lem2} in \eqref{def:I_1} and using part \ref{th:BEa}, we get
\begin{align*}
\int_0^{\eta \bnym}
 & \sup_{0 \leqp \theta \leqp u} I_1(u, \theta) du \\
 & \leqp c_n^{-1} C_0 (l_{1,n}+l_{2,n}) \int_0^{\eta \bnym}  \int_{A_1}  (|s|+|u|+1)^3  e^{-(s^2+u^2)/24} ds du \\
 & \leqp N_n^{-1/2} \tilde{c}_5^{-1} C_0 (c_2^3 + c_4^3) \int_{\R^2} (|s|+|u|+1)^3  e^{-(s^2+u^2)/24}dsdu
\end{align*}  
and the result follows with
\[
C_1 = \tilde{c}_5^{-1} C_0 (c_2^3 + c_4^3) \int_{\R^2} (|s|+|u|+1)^3  e^{-(s^2+u^2)/24} ds du.
\]
\end{proof}

Now, we study the integral on $A_2$. 

\begin{lemma}\label{lem:I2}
There exist positive constants $C_2$ and $C_3$, only depending on $\tilde{c}_1$, $c_1$, $c_2$, $\tilde{c}_3$, $c_3$, $c_4$, $c_5$, $\tilde{c}_5$, and $c_6$, such that
\begin{equation}\label{eq:I2}
\int_0^{\eta \bnym} \sup_{0\leqp \theta \leqp t} I_2(u, \theta)du \leqp C_2 e^{- C_3 N_n}.
\end{equation}
\end{lemma}

\begin{proof} We use the controls \eqref{eq:maj_der_phi_n}, \eqref{maj_expo_st}, and $\abs{\varphi_n} \leqslant 1$ to get
\begin{align*}
&\abs{\left(\frac{\partial}{\partial t} \left[e^{t^2/2}\varphi_n^{N_n}\left(\frac{s}{\bnxm},\frac{t}{\bnym}\right)\right]\right)_{t=\theta}}\\
& = e^{\theta^2/2}\left|\varphi_n^{N_n-1}\left(\frac{s}{\bnxm},\frac{\theta}{\bnym}\right)\right| \cdot\left|\theta \varphi_n\left(\frac{s}{\bnxm},\frac{\theta}{\bnym}\right) \right. \\
& \hspace{7cm} \left. + \frac{N_n}{\ET{Y\chap} N_n^{1/2}} \frac{\partial \varphi_n}{\partial t} \left(\frac{s}{\bnxm},\frac{\theta}{\bnym}\right)\right|\\
&\leqp  e^{\theta^2/2} e^{-(s^2+\theta^2) \cdot c_5(N_n-1)/N_n} (\abs{s} + 2\abs{\theta}).
\end{align*}
Finally by \eqref{def:I_2} and for $N_n \geqslant 2$, we conclude that
\begin{align*}
 & \int_0^{\eta \bnym} \sup_{0 \leqp \theta \leqp u} I_2(u, \theta) du \\
 & \leqp 2 c_n^{-1} \int_0^{+\infty} \int_{\varepsilon \ET{X\chap}N_n^{1/2}}^{+\infty} \sup_{0 \leqslant \theta \leqslant u} \bigg[ (s + 2\theta) \exp \bigg(- \frac{u^2}{2} + \frac{\theta^2}{2} \bigg( 1 - 2c_5\frac{N_n-1}{N_n} \bigg) \bigg) \bigg] \\
 & \hspace{10cm}\cdot e^{-s^2 \cdot c_5(N_n-1)/N_n} ds du \\
 & \leqslant 2 \tilde{c}_5^{-1} \int_0^{+\infty} \int_{\varepsilon \ET{X\chap}N_n^{1/2}}^{+\infty} (s + 2t) e^{-\min(1, c_5)u^2/2} e^{-s^2 c_5/2} ds dt \\
 & \leqslant 2 \tilde{c}_5^{-1} \frac{2}{c_5}e^{-N_n c_5 \varepsilon^2 \ET{X\chap}^2/2} \frac{\sqrt{2\pi}}{2\sqrt{\min(1, c_5)}} + 2 \tilde{c}_5^{-1} \frac{2}{\min(1, c_5)} \frac{e^{-N_n c_5 \varepsilon^2 \ET{X\chap}^2/2}}{c_5\varepsilon \ET{X\chap} N_n^{1/2}}.
\end{align*}
The conclusion follows with
\begin{equation} \label{C2}
C_2 \defeq 2 \tilde{c}_5^{-1}c_5^{-1} \left( \frac{\sqrt{2\pi}}{\sqrt{\min(1, c_5)}} + \frac{2}{\min(1, c_5) \min \bigg( \frac{2}{9} c_1 c_2^3, \pi \bigg) \tilde{c}_1} \right)
\end{equation}
and
\begin{equation} \label{C3}
C_3 \defeq c_5 \min \bigg( \frac{2}{9} c_1 c_2^3, \pi \bigg)^2 \tilde{c}_1^2/2.
\end{equation}
\end{proof}

To conclude to part b) of Theorem \ref{th:BE_cond_strong}, just wright
\[
C_2 e^{-C_3 N_n} = \frac{C_2 C_3^{-1/2}}{N_n^{1/2}} (C_3 N_n)^{1/2} e^{- C_3 N_n} \leqslant \frac{C_2 C_3^{-1/2}}{N_n^{1/2}} (1/2)^{1/2} e^{-1/2},
\]
since $x^{1/2} e^{-x}$ is maximum in $1/2$. So,
\[
\sup_x \left|\Prob\left(\frac{U_n-N_n \Espe\left[Y^{(n)}\right]}{N_n^{1/2}\tau_n}\leqp x \right)-\Phi(x)\right| \leqp \frac{C}{N_n^{1/2}}
\]
with
\begin{equation} \label{C1}
C \defeq C_1 + C_2 C_3^{-1/2} (1/2)^{1/2} e^{-1/2}.
\end{equation}
\end{proof}

\begin{proof}[Proof of Theorem \ref{th:BE_cond_strong} - Part c)]
We start proving \eqref{eq:moment}. We adapt the proof given in \cite{Janson01}. Using \eqref{def:psi_n} with $\Espe[Y\chap] = 0$, and differentiating under the integral sign of \eqref{eq:bartlett}, we naturally have
\begin{align}
 &\abs{\Espe\left[U_n\right]} = \abs{\frac{-i\psi_n'(0)}{2\pi \Prob(S_n=k_n)}} \nonumber\\
 & \leqslant  \frac{\bnx N_n}{2\pi \Prob(S_n=k_n)}  \int_{-\pi \bnxm}^{\pi\bnxm} \abs{\frac{\partial \varphi_n}{\partial t}\left(\frac{s}{\bnxm},0\right)} \cdot \abs{\varphi_n^{N_n-1}\left(\frac{s}{\bnxm},0\right)} ds.\label{eq:exp_U_n}
\end{align}
Using inequality \eqref{esp1b_array} of Lemma \ref{lem:maj_der_phi_n} with $r_n = 0$ and $t = 0$, assumptions \ref{ass:var_X}, \ref{ass:rho_X}, and \ref{ass:rho_Y}, we deduce
\[
\abs{\frac{\partial \varphi_n}{\partial t}\left(\frac{s}{\bnxm},0\right)} \leqp \frac{s^2}{2}\frac{\rho_{Y\chap}^{1/3}\rho_{X\chap}^{2/3}}{\sigma_{X\chap}^2 N_n} \leqp \frac{c_2^2 c_3 c_4}{2 N_n} s^2.
\]
Then using inequality \ref{maj_expo_st} of Lemma \ref{lem:maj_expo} with $t = 0$ and for $N_n \geqslant 2$,
\[
\int_{-\pi \bnxm}^{\pi\bnxm} \abs{\frac{\partial \varphi_n}{\partial t}\left(\frac{s}{\bnxm},0\right)} \cdot \abs{\varphi_n^{N_n-1}\left(\frac{s}{\bnxm},0\right)} ds
\leqp \frac{c_2^2 c_3 c_4}{2 N_n} \int_{\R} s^2 e^{-c_5 s^2/2} ds.
\]
So, \ref{eq:moment} holds with
\begin{equation}\label{c7}
c_7 \defeq \frac{c_2^2 c_3 c_4}{2 \tilde{c}_5} \int_{\R} s^2 e^{-c_5 s^2/2} ds.
\end{equation}

To prove \eqref{eq:moment2}, since $\tau_n = \ET{Y\chap}$ and $\Espe\left[U_n\right]$ is bounded, it suffices to show that the quantity $\abs{\Espe\left[U_n^2\right] - N_n \ET{Y\chap}^2}$ is bounded by some $c_8' N_n^{1/2}$. Proceeding as previously,
\begin{align}
 & \Espe\left[U_n^2\right] = \frac{-\psi_n''(0)}{2\pi \Prob(S_n=k_n)}\nonumber\\
 & = - c_n^{-1} N_n(N_n-1)  \int_{-\pi \bnxm}^{\pi\bnxm} \left(\frac{\partial \varphi_n}{\partial t}\left(\frac{s}{\bnxm},0\right)\right)^2\varphi_n^{N_n-2}\left(\frac{s}{\bnxm},0\right)ds\label{eq:moment3}\\
 & \quad - c_n^{-1} N_n  \int_{-\pi \bnxm}^{\pi\bnxm} \frac{\partial^2 \varphi_n}{\partial t^2}\left(\frac{s}{\bnxm},0\right)\varphi_n^{N_n-1}\left(\frac{s}{\bnxm},0\right)ds.\label{eq:moment4}
\end{align}

First, by inequality \eqref{esp1b_array} with $r_n = 0$ and $t = 0$, the control \eqref{maj_expo_st} with $t = 0$, and for $N_n \geqslant 3$, one has 
\begin{align*}
 & \int_{-\pi \bnxm}^{\pi\bnxm} \abs{\frac{\partial \varphi_n}{\partial t}\left(\frac{s}{\bnxm},0\right)}^2 \abs{\varphi_n^{N_n-2}\left(\frac{s}{\bnxm},0\right)} dv \\
 & \qquad \leqp \frac{c_2^4 c_3^2 c_4^2}{4 N_n^2} \int_{\R} s^4 e^{- c_5 s^2/3}ds,
\end{align*}
and finally using \ref{th:BEa}, the term \eqref{eq:moment3} is bounded by
\begin{equation}\label{c8b}
c_8'' \defeq \frac{c_2^4 c_3^2 c_4^2}{4 \tilde{c}_5} \int_{\R} s^4 e^{- c_5 s^2/3}ds.
\end{equation}

Second, we study the term \eqref{eq:moment4}. We want to show that 
\[
\Delta_n: =  c_n^{-1} \int_{-\pi \bnxm}^{\pi\bnxm} \frac{\partial^2 \varphi_n}{\partial t^2}\left(\frac{s}{\bnxm}, 0\right) \varphi_n^{N_n-1}\left(\frac{s}{\bnxm}, 0\right)ds + \sigma_{Y\chap}^2
\]
is bounded by some $c_8'''/N_n^{1/2}$. Recall that, by Lemma \ref{lem:bartlett} and assumption \ref{ass:exp_X},
\[
\int_{-\pi \bnxm}^{\pi\bnxm} \varphi_n^{N_n}\left(\frac{s}{\bnxm},0\right) dv = 2 \pi \Prob(S_n = k_n) \ET{X\chap} N_n^{1/2} = c_n,
\]
so
\begin{align*}
\Delta_n & = c_n^{-1} \int_{-\pi \bnxm}^{\pi\bnxm}  \left( \frac{\partial^2 \varphi_n}{\partial t^2}\left(\frac{s}{\bnxm},0\right) + \sigma_{Y\chap}^2 \varphi_n\left(\frac{s}{\bnxm},0\right)\right) \\
 & \hspace{8cm} \cdot\varphi_n^{N_n-1} \left(\frac{s}{\bnxm},0\right)ds \\
 &=  c_n^{-1} \int_{-\pi \bnxm}^{\pi\bnxm}  \Espe\bigg[ {Y\chap}^2 \Big( - e^{is\bnx(X\chap-\Espe[X\chap])} \\
 & \hspace{5cm} + \Espe\Big[ e^{is\bnx(X\chap-\Espe[X\chap])} \Big] \Big) \bigg] \\
 & \hspace{8cm} \cdot\varphi_n^{N_n-1} \left(\frac{s}{\bnxm},0\right)ds.
\end{align*}
Applying Taylor theorem to the function
\[
f(s) = - e^{is\bnx(X\chap-\Espe[X\chap])} +\Espe\Big[ e^{is\bnx(X\chap-\Espe[X\chap])} \Big]
\]
yields
\begin{align*}
\abs{f(s)} & \leqslant \abs{s} \sup_{u \in [0, s]} \left\lvert - i\frac{X\chap-\Espe[X\chap]}{\bnxm} e^{iu\bnx(X\chap-\Espe[X\chap])} \right. \\
 & \qquad\qquad\qquad \left. + \Espe\bigg[ i\frac{X\chap-\Espe[X\chap]}{\bnxm} e^{iu\bnx(X\chap-\Espe[X\chap])} \bigg]\right\rvert \\
 & \leqslant \frac{\abs{s}}{N_n^{1/2}} \bigg( \abs{ \frac{X\chap-\Espe[X\chap]}{\sigma_{X\chap}}} + \Espe\bigg[ \abs{ \frac{X\chap-\Espe[X\chap]}{\sigma_{X\chap}}} \bigg] \bigg).
\end{align*}
Thus, using H\"older Inequality,
\begin{align*}
\abs{\Espe[{Y\chap}^2 f(s)]} & \leqslant \frac{\abs{s}}{N_n^{1/2}} \Espe\bigg[ {Y\chap}^2 \bigg( \abs{\frac{X\chap-\Espe[X\chap]}{\sigma_{X\chap}}} + \Espe\bigg[ \abs{\frac{X\chap-\Espe[X\chap]}{\sigma_{X\chap}}} \bigg] \bigg) \bigg] \\
 & \leqslant \frac{\sigma_{Y\chap}^2 \abs{s}}{N_n^{1/2}} \bigg( \frac{\rho_{Y\chap}^{2/3}}{\sigma_{Y\chap}^2} \frac{\rho_{X\chap}^{1/3}}{\sigma_{X\chap}} + 1 \bigg)
\end{align*}
and, applying equation \ref{th:BEa}, assumptions \ref{ass:var_X}, \ref{ass:rho_X}, \ref{ass:var_Y}, \ref{ass:rho_Y}, and the majoration \eqref{maj_expo_st} with $t = 0$, we get
\[
\abs{\Delta_n} \leqslant \frac{\sigma_{Y\chap} }{N_n^{1/2}c_n}  \bigg( \frac{\rho_{Y\chap}^{2/3}}{\sigma_{Y\chap}^2} \frac{\rho_{X\chap}^{1/3}}{\sigma_{X\chap}} + 1 \bigg) \int_{\R} \abs{s} e^{-s^2 c_5 (N_n-1)/N_n} ds \leqslant \frac{c_8'''}{N_n^{1/2}}
\]
with
\begin{equation}\label{c8c}
c_8''' \defeq c_3 \tilde{c}_5^{-1}(1 + c_2 c_4^2) \int_{\R} \abs{s} e^{-s^2 c_5/2} ds.
\end{equation}
Finally,
\[
\abs{\Var(U_n) - N_n \tau_n^2} \leqslant c_7 + c_8'' + c_8''' N_n^{1/2} \leqslant c_8 N_n^{1/2}
\]
with
\begin{align}\label{c8}
 & c_8 \defeq c_7 + c_8'' + c_8''' \nonumber \\
 & = \frac{c_2^2 c_3 c_4}{2 \tilde{c}_5} \int_{\R} s^2 e^{-c s^2/2} ds + \frac{c_2^4 c_3^2 c_4^2}{4 \tilde{c}_5} \int_{\R} s^4 e^{- c_5 s^2/3}ds + c_3 \tilde{c}_5^{-1}(1 + c_2 c_4^2) \int_{\R} \abs{s} e^{-s^2 c_5/2} ds.
\end{align}

Now we turn to the proof of \eqref{eq:be_2}. Let us show that the previous estimates of $\Espe[U_n]$ and $\Var(U_n)$ make it possible to apply \eqref{eq:be}. Remind that $\Espe\left[Y\chap\right]=0$. Write
\[
\left\{ \frac{U_n - \Espe[U_n]}{\Var\left(U_n\right)^{1/2}} \leqp x \right\} = \left\{ \frac{U_n}{N_n^{1/2} \sigma_{Y\chap}} \leqp a_n x + b_n \right\},
\]
where
\[
a_n \defeq \frac{\Var(U_n)^{1/2}}{N_n^{1/2} \sigma_{Y\chap}} \quad \text{and} \quad b_n \defeq \frac{\Espe[U_n]}{N_n^{1/2} \sigma_{Y\chap}}.
\]
The previous estimates of $\Espe[U_n]$ and $\Var(U_n)$ yield
\[
\abs{a_n - 1} \leqslant \abs{a_n^2 - 1} \leqslant c_8 \tilde{c}_3^{-1} N_n^{-1/2} \quad \text{and} \quad b_n \leqslant c_7 \tilde{c}_3^{-1} N_n^{-1/2}.
\]
Now, 
\begin{align*}
\abs{\Prob\left( \frac{U_n - \Espe[U_n]}{\Var\left(U_n\right)^{1/2}} \leqslant x \right) - \Phi(x)}
 & \leqslant \abs{\Prob\left( \frac{U_n}{N_n^{1/2} \sigma_{Y\chap}} \leqp a_n x + b_n \right) - \Phi(a_n x + b_n)} \\
 & \hspace{5cm} + \abs{\Phi(a_n x + b_n) - \Phi(x)} \\
 & \leqslant \frac{C_1}{N_n^{1/2}} + C_2 e^{-C_3 N_n} + \abs{\Phi(a_n x + b_n) - \Phi(x)}.
\end{align*}
For $N_n > 4 c_8^2/\tilde{c}_3^2$, $a_n \geqslant 1/2$ and applying Taylor theorem to $\Phi$ yields
\begin{align*}
\abs{\Phi(a_n x + b_n) - \Phi(x)} & \leqslant \abs{(a_n - 1) x + b_n} \sup_t \frac{e^{-t^2/2}}{\sqrt{2\pi}} \\
 & \leqslant N_n^{-1/2} \max(c_8 \tilde{c}_3^{-1}, c_7 \tilde{c}_3^{-1}) (\abs{x} + 1) e^{-(|x|/2 - c_7 \tilde{c}_3^{-1})^2/2},
\end{align*}
the supremum being over $t$ between $x$ and $a_n x + b_n$. The last function in $x$ being bounded, we get \eqref{eq:be_2} with
\[
\tilde{C}_1 \defeq \max(c_8 \tilde{c}_3^{-1}, c_7 \tilde{c}_3^{-1}) \sup_{x \in \R} \Big[(\abs{x} + 1) e^{-(|x|/2 - c_7 \tilde{c}_3^{-1})^2/2} \Big].
\]
\end{proof}

\subsection{Proof of Theorem \ref{th:nagaev_weak_array_mob_v2}}

We start with the proof of Theorem \ref{th:nagaev_weak_array_mob_v2}, which relies on three different lemmas. 
 

\begin{proof}[Proof of Theorem \ref{th:nagaev_weak_array_mob_v2}] Let $z_n$ such that $\displaystyle{\liminf_{n \to \infty} \frac{z_n}{N_n} > 0}$. Since $Y^{(n)} - \Espe\left[Y^{(n)}\right]$ also satisfies the hypotheses, we can assume that $\Espe\left[Y^{(n)}\right] = 0$. Define
\[
P_{N_n} = \Prob( T_n \geqslant z_n )
\]
and for any $m\in \intervallentff{0}{{N_n}}$,
\[
P_{{N_n},m} = \Prob\left( T_n \geqslant z_n, \quad \forall i \in \intervallentff{1}{{N_n}-m} \; Y^{(n)}_i < z_n, \quad \forall i \in \intervallentff{{N_n}-m+1}{{N_n}} \; Y^{(n)}_i \geqslant z_n \right)
\]
with the usual convention $\intervallentff{1}{0}=\emptyset$ and $\intervallentff{N_n+1}{N_n}=\emptyset$.
Now write
\begin{equation}\label{eq:decomp}
P_{N_n} = P_{{N_n},0} + {N_n} P_{{N_n},1} + \sum_{m=2}^{N_n} \binom{{N_n}}{m} P_{{N_n},m}.
\end{equation}

Using Lemmas \ref{lem1_weak_array_mob_v2}, \ref{lem2_weak_array_mob_v2} and \ref{lem3_weak_array_mob_v2} that follow, we conclude the proof of Theorem \ref{th:nagaev_weak_array_mob_v2}.
\end{proof}

\begin{lemma}\label{lem1_weak_array_mob_v2}
\[
\limsup_{n \to \infty} \frac{1}{\sqrt{z_n}}\log(P_{{N_n},0}) \leqslant -\alpha.
\]
\end{lemma}

\begin{lemma}\label{lem2_weak_array_mob_v2}
\[
-\beta \leqslant \liminf_{n \to \infty} \frac{1}{\sqrt{z_n}} \log({N_n} P_{{N_n},1}) \leqslant \limsup_{n \to \infty} \frac{1}{\sqrt{z_n}} \log({N_n} P_{{N_n},1}) \leqslant -\alpha.
\]
\end{lemma}

\begin{lemma}\label{lem3_weak_array_mob_v2}
\[
\sum_{m=2}^{N_n} \binom{{N_n}}{m} P_{{N_n},m} = o \left( e^{-\alpha \sqrt{z_n}} \right).
\]
\end{lemma}

\begin{proof}[Proof of Theorem \ref{th:nagaev_weak_array_mob_v2}]
Lemmas \ref{lem1_weak_array_mob_v2}, \ref{lem2_weak_array_mob_v2}, and \ref{lem3_weak_array_mob_v2} yield, for all $\alpha' < \alpha$,
\begin{align*}
-\beta
 & \leqslant \liminf_{n \to \infty} \frac{1}{\sqrt{z_n}} \log({N_n}P_{{N_n},1}) \leqslant \liminf_{n \to \infty} \frac{1}{\sqrt{z_n}} \log(P_{N_n})\\
 & \leqslant \limsup_{n \to \infty} \frac{1}{\sqrt{z_n}} \log(P_{N_n})
 \leqslant \lim_{n \to \infty} \frac{1}{\sqrt{z_n}} \log \left( 3 e^{-\alpha' \sqrt{z_n}} \right)
  = -\alpha'.
\end{align*}
Conclude by letting $\alpha' \to \alpha$.
\end{proof}

\begin{proof}[Proof of Lemma \ref{lem3_weak_array_mob_v2}]
Let $\alpha' \in \intervalleoo{\alpha/2}{\alpha}$. Using \eqref{hyp:queue2_weak_array_mob_v2} and noting that $z_n \geqslant \sqrt{z_n}$ for $n$ large enough, we have, for all $n$ large enough,
\[
\sum_{m=2}^{N_n} \binom{{N_n}}{m} P_{{N_n},m} \leqslant \sum_{m=2}^{N_n} {N_n}^m \Prob(Y^{(n)}_1 \geqslant z_n)^m \leqslant \frac{{N_n}^2 e^{-2 \alpha' \sqrt{z_n}}}{1 - {N_n} e^{-\alpha' \sqrt{z_n}}} = o \left( e^{-\alpha \sqrt{z_n}} \right).
\]
\end{proof}

\begin{proof}[Proof of Lemma \ref{lem2_weak_array_mob_v2}]
First, using \eqref{hyp:queue2_weak_array_mob_v2},
\[
\limsup_{n \to \infty} \frac{1}{\sqrt{z_n}} \log( {N_n} P_{{N_n},1} ) \leqslant \limsup_{n \to \infty} \frac{1}{\sqrt{z_n}} \log \Prob(Y^{(n)} \geqslant z_n) \leqslant - \alpha .
\]
Let us prove the converse inequality. Let $\varepsilon > 0$. We have
\begin{align*}
P_{{N_n},1} & = \Prob\left( T_n \geqslant z_n, \quad Y^{(n)}_{N_n} \geqslant z_n, \quad \forall i \in \intervallentff{1}{{N_n}-1} \quad Y^{(n)}_i < z_n \right)\\
 & = \int_{z_n}^{+\infty} \Prob\left( T_{n-1} \geqslant z_n - u, \quad \forall i \in \intervallentff{1}{{N_n}-1} \quad Y^{(n)}_i < z_n \right) \Prob(Y^{(n)} \in du)\\
 & \geqslant \int_{z_n + {N_n}\varepsilon}^{+\infty} \Prob\left( T_{n-1} \geqslant z_n - u, \quad \forall i \in \intervallentff{1}{{N_n}-1} \quad Y^{(n)}_i < z_n \right) \Prob(Y^{(n)} \in du)\\
 & \geqslant \Prob\left( T_{n-1} \geqslant -{N_n}\varepsilon, \quad \forall i \in \intervallentff{1}{{N_n}-1} \quad Y^{(n)}_i < z_n \right) \Prob(Y^{(n)} \geqslant z_n + {N_n}\varepsilon).
\end{align*}
Observe that
\begin{align*}
\Prob\big( T_{n-1} \geqslant -{N_n}\varepsilon, \quad & \forall i \in \intervallentff{1}{{N_n}-1} \quad Y^{(n)}_i < z_n \big)\\
 & \geqslant \Prob\left( Y^{(n)} < z_n \right)^{{N_n}-1} - \Prob\left( T_{n-1} < -{N_n}\varepsilon \right) \to 1.
\end{align*}
Indeed, $\Prob\left( Y^{(n)}_1 < z_n \right)^{{N_n}-1} \to 1$, using \eqref{hyp:queue2_weak_array_mob_v2}; and, by Chebyshev inequality and assumption \ref{hyp:mt2_weak_array_mob_v2}
, 
\[
\Prob\left( T_{n-1} < -{N_n}\varepsilon \right) \leqslant \frac{\sigma_{Y^{(n)}}^2}{{N_n}\varepsilon^2}\to 0,
\]
the random variables $Y^{(n)}$ being assumed centered. Finally, using \eqref{hyp:queue1_weak_array_mob_v2} and \ref{hyp:z_weak_array_mob_v2}, and noting $\delta = \displaystyle{\liminf_{n \to \infty} \frac{z_n}{N_n}}$, one gets
\begin{align*}
\liminf_{n \to \infty} \frac{1}{\sqrt{z_n}} \log ({N_n} P_{{N_n},1})
 & \geqslant \liminf_{n \to \infty} \sqrt{\frac{z_n + {N_n}\varepsilon}{z_n}} \frac{1}{\sqrt{z_n + {N_n}\varepsilon}} \log \Prob(Y^{(n)} \geqslant z_n + {N_n}\varepsilon) \\
 & \geqslant -\beta \sqrt{\frac{\delta + \varepsilon}{\delta}}.
\end{align*}
Conclude by letting $\varepsilon \to 0$.
\end{proof}

\begin{proof}[Proof of Lemma \ref{lem1_weak_array_mob_v2}]
Let $\alpha' \in \intervalleoo{0}{\alpha}$ and $s_n = \alpha'/\sqrt{z_n}$. The exponential Chebyshev inequality for $T_n$ conditioned on $\{ \forall i \in \intervallentff{1}{{N_n}}, \, Y^{(n)}_i < z_n \}$ yields
\[
P_{{N_n},0} \leqslant e^{-s_n z_n} \Espe\left[ e^{s_n Y^{(n)}} \indic_{Y^{(n)} < z_n} \right]^{N_n}.
\]
If we prove that
\[
\Espe\left[ e^{s_n Y^{(n)}} \indic_{Y^{(n)} < z_n} \right] = 1 + o \left( \frac{1}{N_n^{1/2}} \right),
\]
then
\[
\log(P_{{N_n},0}) \leqslant -\alpha' \sqrt{z_n} + o(N_n^{1/2})
\]
and the conclusion follows by letting $\alpha' \to \alpha$. Let $\eta \in ]3/4, 1[$. Write
\begin{align*}
\Espe &\left( e^{s_n Y^{(n)}} \indic_{Y^{(n)} < z_n} \right)\\
 & = \int_{-\infty}^{\sqrt{z_n}} e^{s_n u} \Prob(Y^{(n)} \in du) + \int_{\sqrt{z_n}}^{z_n - (z_n)^\eta} e^{s_n u} \Prob(Y^{(n)} \in du) + \int_{z_n - (z_n)^\eta}^{z_n} e^{s_n u} \Prob(Y^{(n)} \in du)\\
 & \eqdef I_1 + I_2 + I_3.
\end{align*}

By a Taylor expansion of $f(t)=e^t$, \ref{hyp:mt2_weak_array_mob_v2} and \ref{hyp:z_weak_array_mob_v2}, there exists
\[
\theta(u) \leqp s_n u \leqp s_n \sqrt{z_n} = \alpha'
\]
such that
\begin{align*}
I_1 & \leqslant \int_{-\infty}^{\sqrt{z_n}} \left(1 + s_n u + \frac{s_n^2 u^2}{2}e^{\theta(u)}\right) \Prob(Y^{(n)} \in du) \\
 & \leqslant \int_{-\infty}^{+\infty} \left(1 + s_n u + \frac{s_n^2 u^2}{2}e^{\alpha'}\right) \Prob(Y^{(n)} \in du) = 1 + 0 + \frac{\alpha'^2 \sigma_{Y^{(n)}}^2}{2z_n} e^{\alpha'}
= 1 + o\left( \frac{1}{N_n^{1/2}} \right).
\end{align*}
Let $n_0$ such that, for all $n \geqslant n_0$ and $u \geqslant \sqrt{z_n}$, $\log \Prob(Y^{(n)} \geqslant u) \leqslant -\alpha' \sqrt{u}$. Suppose $n$ is larger than $n_0$. Integrating by part, we get
\begin{align*}
I_2 & = - \Big[ e^{s_n u} \Prob(Y^{(n)} \geqslant u) \Big]_{\sqrt{z_n}}^{z_n-(z_n)^\eta} + s_n \int_{\sqrt{z_n}}^{z_n-(z_n)^\eta} e^{s_n u} \Prob(Y^{(n)} \geqslant u) du\\
 & \leqslant e^{s_n \sqrt{z_n}} \Prob(Y^{(n)} \geqslant \sqrt{z_n}) + s_n \int_{\sqrt{z_n}}^{z_n-(z_n)^\eta} e^{s_n u - \alpha' \sqrt{u}} du \\
 & \leqslant e^{\alpha' (1-(z_n)^{1/4})} + s_n \int_{\sqrt{z_n}}^{z_n-(z_n)^\eta} \exp \left( \alpha' \left( \frac{u}{\sqrt{z_n}} - \sqrt{u}\right) \right) du.
\end{align*}
Since, for all $t \in [0, 1]$, $\sqrt{1-t} \leqslant 1 - t/2$, we get, for all $u \in [\sqrt{z_n}, z_n-(z_n)^\eta]$ and $n$ large enough to have $\left(z_n\right)^{\nu-1}<1$,
\[
\frac{u}{\sqrt{z_n}} - \sqrt{u} \leqslant \sqrt{u} \left( \sqrt{1 - (z_n)^{\eta-1}} - 1 \right) \leqslant -\frac{(z_n)^{\eta-3/4}}{2}.
\]
Hence,
\[
I_2 = o\left( \frac{1}{N_n^{1/2}} \right).
\]
Let $\alpha'' \in \intervalleoo{\alpha'}{\alpha \wedge 2\alpha'}$. Let $n_1$ such that, for all $n \geqslant n_1$ and $u \geqslant z_n -z_n^\eta$, $\log \Prob(Y^{(n)} \geqslant u) \leqslant -\alpha'' \sqrt{u}$. Suppose $n$ is larger than $n_1$. Integrating by part, we get
\begin{align*}
I_3 & = - \Big[ e^{s_n u} \Prob(Y^{(n)} \geqslant u) \Big]_{z_n - z_n^\eta}^{z_n} + s_n \int_{z_n-z_n^\eta}^{z_n} e^{s_n u} \Prob(Y^{(n)} \geqslant u) du\\
 & \leqslant e^{s_n (z_n-z_n^\eta)} \Prob(Y^{(n)} \geqslant z_n - z_n^\eta) + s_n \int_{z_n-z_n^\eta}^{z_n} e^{s_n u - \alpha'' \sqrt{u}} du.
\end{align*}
Now, since $\sqrt{t} \geqslant t$ if $t \in [0, 1]$,
\begin{align*}
e^{s_n (z_n-z_n^\eta)} \Prob(Y^{(n)} \geqslant z_n - z_n^\eta)
 & \leqslant \exp \left( \sqrt{z_n} \left( \alpha' \left(1-z_n^{\eta-1}\right) - \alpha'' \left(1-z_n^{\eta-1}\right)^{1/2} \right) \right)\\
 & \leqslant \exp \left( \sqrt{z_n} (\alpha' - \alpha'') (1-z_n^{\eta-1}) \right)  = o\left( \frac{1}{N_n^{1/2}} \right).
\end{align*}
Finally, applying Taylor theorem to the function $f(u) = s_n u - \alpha'' \sqrt{u}$ around the point $z_n$ yields
\[
f(u) = \frac{\alpha' u}{\sqrt{z_n}} - \alpha'' \sqrt{u} = (\alpha' - \alpha'') \sqrt{z_n} + \left( \frac{\alpha'}{\sqrt{z_n}} - \frac{\alpha''}{2\sqrt{c}} \right) (u - z_n)
\]
with $c \in [u, z_n]$. Since $\alpha'' < 2\alpha'$, we have
\[
\left( \frac{\alpha'}{\sqrt{z_n}} - \frac{\alpha''}{2\sqrt{c}} \right) (u - z_n) \leqslant \left( \frac{\alpha'}{\sqrt{z_n}} - \frac{\alpha''}{2\sqrt{z_n - z_n^\eta}} \right) (u - z_n) \leqslant 0,
\]
for $n$ large enough and we conclude that
\[
I_3 = o\left( \frac{1}{N_n^{1/2}} \right).
\]
\end{proof}

\subsection{Proof of Theorem \ref{th:nagaev_weak_array_cond_mob_cor}}

Now we turn to the proof of Theorem \ref{th:nagaev_weak_array_cond_mob_cor}. So as to apply Theorem \ref{th:nagaev_weak_array_mob_v2}, we need the next result, which is analogous to equation \eqref{eq:moment}.

\begin{proposition}\label{moy_weak_array_cond_mob_cor}
Under assumptions \ref{hyp:Xmt2_weak_array_cond_mob_cor}, \ref{hyp:Xphi_weak_array_cond_mob_cor} and \ref{hyp:Ymt2_weak_array_cond_mob_cor}, one has
\[
\Espe\left[ \left. T_n \right| S_n = k_n\right] = {N_n}\Espe\left[Y^{(n)}\right] + o({N_n}).
\]
\end{proposition}

\begin{proof}
Using inequality \eqref{eq:exp_U_n} and Proposition \ref{tll_weak_array_cond_mob_cor} yield
\begin{align}
 & \abs{\Espe\left[\left.T_n - {N_n}\Espe\left[Y^{(n)}\right] \right| S_n = k_n\right]}  = \abs{\frac{-i \psi_n'(0)}{2\pi \Prob(S_n = k_n)}} \nonumber \\
 & \leqp \frac{{N_n}}{2\pi m} \int_{-\pi \ET{X^{(n)}} N_n^{1/2}}^{\pi \ET{X^{(n)}} N_n^{1/2}} \abs{\frac{\partial \varphi_n}{\partial t}\left(\frac{s}{\ET{X^{(n)}} N_n^{1/2}}, 0\right)} \cdot \abs{\varphi_n^{{N_n}-1}\left(\frac{s}{\ET{X^{(n)}} N_n^{1/2}}, 0\right)} ds. \label{espb_array}
\end{align}
It remains to show that the integral converges to $0$. Putting together \eqref{espb_array} and \eqref{esp1b_array}, and using hypothesis \ref{hyp:Ymt2_weak_array_cond_mob_cor} and the control \eqref{maj_expo_s}, one gets
\[
\Espe\left[\left.T_n - {N_n}\Espe\left[Y^{(n)}\right] \right| S_n = k_n\right] = o({N_n}).
\]
\end{proof}

\begin{proof}[Proof of Theorem \ref{th:nagaev_weak_array_cond_mob_cor}]
Let $y > 0$. Since $(X^{(n)}, Y^{(n)} - \Espe\left[Y^{(n)}\right])$ also satisfies the hypotheses, we can assume that $\Espe\left[Y^{(n)}\right] = 0$. According to Proposition \ref{moy_weak_array_cond_mob_cor},
\[
y_n \defeq y + \frac{1}{{N_n}}\Espe\left[U_n\right] \to y.
\]

We have
\begin{align*}
\Prob( U_n - \Espe\left[ U_n \right] \geqslant {N_n}y )
 & = \Prob( T_n - \Espe\left[ T_n | S_n = k_n \right] \geqslant {N_n}y | S_n = k_n ) \\
 & = \frac{ \Prob( T_n \geqslant {N_n}y_n, \, S_n = k_n )}{\Prob( S_n = k_n )} \leqslant \frac{\Prob(T_n \geqslant {N_n}y_n)}{\Prob( S_n = k_n )}.
\end{align*}
The conclusion follows using Theorem \ref{th:nagaev_weak_array_mob_v2}, Proposition \ref{tll_weak_array_cond_mob_cor} and \ref{hyp:Xmt2_weak_array_cond_mob_cor}.

Using decomposition \eqref{eq:decomp}, we get
\begin{align*}
\Prob  ( U_n &- \Espe\left[ U_n \right] \geqslant {N_n}y )=\Prob(T_n - \Espe\left[ T_n | S_n = k_n \right] \geqslant {N_n}y | S_n = k_n )\\
 & = \frac{\Prob( T_n \geqslant {N_n}y_n, \, S_n = k_n )}{\Prob( S_n = k_n )}  \geqslant \Prob( T_n \geqslant {N_n}y_n, \, S_n = k_n ) \\
 & \geqslant N_n \Prob\big( T_n \geqslant {N_n}y_n, \quad Y^{(n)}_n \geqslant {N_n}y_n, \forall i \in \intervallentff{1}{{N_n}-1} \quad Y^{(n)}_i < {N_n}y_n, \quad S_n = k_n \big).
\end{align*}
Define
\[
Q_{{N_n},1} \defeq \Prob\left( T_n \geqslant {N_n}y_n, \quad Y^{(n)}_n \geqslant {N_n}y_n, \quad \forall i \in \intervallentff{1}{{N_n}-1} \quad Y^{(n)}_i < {N_n}y_n, \quad S_n = k_n \right).
\]
It remains to show that
\[
\liminf_{n \to \infty} \frac{1}{\sqrt{{N_n}y}}\log({N_n} Q_{{N_n},1}) \geqslant -\beta ,
\]
which is analogous to the lower bound of Lemma \ref{lem2_weak_array_mob_v2}. We have, for any $\varepsilon >0$, 
\begin{align*}
 & Q_{{N_n},1} = \Prob\left( T_n \geqslant {N_n}y_n, \quad Y^{(n)}_n \geqslant {N_n}y_n, \quad \forall i \in \intervallentff{1}{{N_n}-1} \quad Y^{(n)}_i < {N_n}y_n, \quad S_n = k_n \right)\\
 & = \int_{{N_n}y_n}^{+\infty} \Prob\left( T_{n-1} \geqslant {N_n}y_n - u, \; \forall i \in \intervallentff{1}{{N_n}-1} \; Y^{(n)}_i < {N_n}y_n, \; S_n = k_n \right) \Prob(Y^{(n)} \in du)\\
 & \geqslant \int_{{N_n}(y_n+\varepsilon)}^{+\infty} \Prob\left( T_{n-1} \geqslant {N_n}y_n - u, \; \forall i \in \intervallentff{1}{{N_n}-1} \; Y^{(n)}_i < {N_n}y_n, \; S_n = k_n \right) \Prob(Y^{(n)} \in du)\\
  & \geqslant \Prob\left( T_{n-1} \geqslant -{N_n}\varepsilon, \; \forall i \in \intervallentff{1}{{N_n}-1} \; Y^{(n)}_i < {N_n}y_n, \; S_n = k_n \right) \Prob(Y^{(n)} \geqslant {N_n}(y_n+\varepsilon)).
\end{align*}
Observe that
\begin{align*}
\Prob\big( T_{n-1} \geqslant -{N_n}\varepsilon, \quad & \forall i \in \intervallentff{1}{{N_n}-1} \quad Y^{(n)}_i < {N_n}y_n, \quad S_n = k_n \big)\\
 & \geqslant \Prob\left( Y^{(n)} < {N_n}y_n\right)^{{N_n}-1} - (1 - \Prob(S_n = k_n)) - \Prob\left( T_{n-1} < -{N_n}\varepsilon \right).
\end{align*}
For $\alpha' \in \intervalleoo{0}{\alpha}$ and $n$ large enough, using \eqref{hyp:Yqueue2_weak_array_cond_mob_cor}, one has
\[
\Prob\left( Y^{(n)} < {N_n}y_n \right)^{{N_n}-1} \geqslant (1 - e^{-\alpha'\sqrt{{N_n}y_n}})^{{N_n}-1} = 1 + o\left( \frac{1}{N_n^{1/2}} \right).
\]
By Chebyshev Inequality and hypothesis \ref{hyp:Ymt2_weak_array_cond_mob_cor}
, one has straightforwardly
\[
\Prob\left( T_{n-1} < -{N_n}\varepsilon\right) \leqslant \frac{\sigma_{Y^{(n)}}^2}{{N_n} \varepsilon^2} = o\left( \frac{1}{N_n^{1/2}} \right).
\]
Hence, using Proposition \ref{tll_weak_array_cond_mob_cor} and hypotheses \ref{hyp:Xmt2_weak_array_cond_mob_cor} and \ref{hyp:Yqueue1_weak_array_cond_mob_cor},
\begin{align*}
\liminf_{n \to \infty} \frac{1}{\sqrt{{N_n}y}} \log ({N_n} Q_{n,1})
 & \geqslant \liminf_{n \to \infty} \frac{1}{\sqrt{{N_n}y}} \log \left( \frac{m}{\ET{X^{(n)}} N_n^{1/2}} \right) + \liminf_{n \to \infty} \frac{1}{\sqrt{{N_n}y}} \log \Prob(Y^{(n)} \\
& \geqslant {N_n}(y_n+\varepsilon)) \geqslant - \beta \sqrt{\frac{y + \varepsilon}{y}}.
\end{align*}
Conclude by letting $\varepsilon \to 0$.
\end{proof}

\bibliographystyle{plain}
\bibliography{biblio_gde_dev}
\end{document}